\documentclass[final]{siamart0516}

\usepackage{booktabs,graphicx}
\usepackage{amsmath,amssymb}
\usepackage{url,hyperref}
\usepackage{xcolor}
\usepackage{xfrac}
\usepackage{multirow,paralist} 
\usepackage{algorithm}
\usepackage{algpseudocode}
\usepackage{pifont}
\usepackage[letterpaper, margin=1.5in]{geometry}
\usepackage{afterpage, enumitem}

\newsiamremark{remark}{Remark}
\newtheorem{assumption}{Assumption}[section]

\newcommand{\Pbb}{\mathbb{P}}
\newcommand{\R}{\mathbb{R}}
\newcommand{\Rext}{\R\cup\{+\infty\}}

\newcommand{\abss}[1]{\vert#1\vert}

\newcommand{\set}[1]{\left\{#1\right\}}
\newcommand{\sets}[1]{\{#1\}}
\newcommand{\norm}[1]{\left\Vert#1\right\Vert}
\newcommand{\norms}[1]{\Vert#1\Vert}

\newcommand{\Eproof}{\hfill $\square$}
\newcommand{\prox}{\mathrm{prox}}

\newcommand{\proj}{\mathrm{proj}}

\newcommand{\relint}[1]{\mathrm{ri}\left(#1\right)}
\newcommand{\argmin}{\mathrm{arg}\!\displaystyle\min}

\newcommand{\dom}[1]{\mathrm{dom}(#1)}

\newcommand{\zero}[1]{{\boldsymbol{0}}}
\newcommand{\expect}[2]{\mathbb{E}_{#1}\left[#2\right]}
\newcommand{\iprod}[1]{\left\langle #1\right\rangle}
\newcommand{\iprods}[1]{\langle #1\rangle}

\newcommand{\ri}[1]{\mathrm{ri}\left(#1\right)}

\newcommand{\Ec}{\mathcal{E}}
\newcommand{\Xc}{\mathcal{X}}
\newcommand{\Yc}{\mathcal{Y}}

\newcommand{\Dc}{\mathcal{D}}
\newcommand{\Lc}{\mathcal{L}}

\newcommand{\Rc}{\mathcal{R}}
\newcommand{\Gc}{\mathcal{G}}
\newcommand{\Tc}{\mathcal{T}}
\newcommand{\Vc}{\mathcal{V}}
\newcommand{\Lb}{\mathbf{L}}

\newcommand{\Pc}{\mathcal{P}}
\newcommand{\Kc}{\mathcal{K}}
\newcommand{\Oc}{\mathcal{O}}

\newcommand{\BigO}[1]{\mathcal{O}\left(#1\right)}

\newcommand{\myeq}[2]{\vspace{-0.5ex}
\begin{equation}\label{#1}
{#2}
\vspace{-0.5ex}
\end{equation}
}
\newcommand{\myeqn}[1]{\vspace{-0.5ex}
\begin{equation*}
{#1}
\vspace{-0.5ex}
\end{equation*}
}

\newcommand{\kdist}[2]{\mathrm{dist}_{#1}\left(#2\right)}

\headers{A New Primal-Dual Algorithm for Compositional Convex Optimization}{Y. Zhu, D. Liu, \text{\lowercase{and}} Q. Tran-Dinh}

\title{A New Primal-Dual Algorithm for A Class of Nonlinear Compositional Convex Optimization Problems}

\author{
Yuzixuan Zhu, Deyi Liu, and Quoc Tran-Dinh\thanks{Department of Statistics and Operations Research, 
The University of North Carolina at Chapel Hill (UNC), 318-Hanes Hall, Chapel Hill, NC27599-3260, USA. 
\newline Corresponding author: {\tt quoctd@email.unc.edu}.} 
}

\date{Received: date / Accepted: date}
    
\begin{document}

\maketitle

\begin{abstract}
We develop a novel primal-dual algorithm to solve a class of nonsmooth and nonlinear compositional convex minimization problems, which covers many existing and brand-new models as special cases.
Our approach relies on a combination of a new nonconvex potential function, Nesterov's accelerated scheme, and an adaptive parameter updating strategy.
Our algorithm is single-loop and has low per-iteration complexity.
Under only general convexity and mild assumptions, our algorithm achieves $\BigO{1/k}$ convergence rates through three different criteria: primal objective residual, dual objective residual, and primal-dual gap, where $k$ is the iteration counter.
Our rates are both \textit{ergodic} (i.e., on an averaging sequence) and \textit{non-ergodic} (i.e., on the last-iterate sequence).
These convergence rates can be accelerated up to $\BigO{1/k^2}$ if only one objective term is strongly convex (or equivalently, its conjugate is $L$-smooth).
To the best of our knowledge, this is the first algorithm achieving optimal rates on the primal last-iterate sequence for nonlinear compositional convex minimization.
As a by-product, we specify our algorithm to solve a general convex cone constrained program with both ergodic and non-ergodic rate guarantees.
We test our algorithms and compare them with two recent methods on a binary classification and a convex-concave game model.

\vspace{1ex}
\end{abstract}

\noindent

\section{Introduction}\label{sec:intro}
The goal of this paper is to develop a novel primal-dual algorithm to solve the following nonsmooth and nonlinear compositional convex minimization problem:
\myeq{eq:primal_cvx}{
	\Pc^{\star} := \min_{x\in \R^p} \Big\{\Pc(x) := F(x) + H\left(g(x)\right)\Big\}, \tag{P}
}
where $F$, $H$, and $g$ satisfy the following assumptions:
\begin{enumerate}
\item $F(x) := f(x) + h(x)$, where $f:\ \R^p \to \R$ is $L_f$-smooth and convex, and $h :\ \R^p \to \R\cup\sets{+\infty}$ is proper, closed, and convex, but not necessarily smooth;
\item $H :\ \R^n \to \R\cup\sets{+\infty}$ is proper, closed, and convex, but not necessarily smooth;
\item $g : \R^p \to \R^n$ is smooth and possibly nonlinear such that $H(g(\cdot))$ is convex.
\end{enumerate}
The corresponding dual problem of \eqref{eq:primal_cvx} is also convex and can be written as
\myeq{eq:dual_cvx}{
	\Dc^{\star} := \min_{y\in\R^n} \Big\{\Dc(y) := G(-y) + H^{*} (y) \Big\}, \quad\text{where} \ G(y) := \max_{x \in \R^p} \big\{ \iprods{y, g(x)} - F(x) \big\} \tag{D}
}
where $H^{*}$ is the Fenchel conjugate of $H$.

Let $\widetilde{\Lc}(x, y) := F(x) + \iprods{g(x), y} - H^{*}(y)$ be the Lagrange function associated with \eqref{eq:primal_cvx}, then \eqref{eq:primal_cvx} and \eqref{eq:dual_cvx} can be written into the following convex-concave minimax form:
\myeq{eq:saddle_point}{
	\min_{x \in \R^p} \max_{y \in \R^n} \Big\{ \widetilde{\Lc}(x, y) := F(x) + \iprods{g(x), y} - H^* (y) \Big\}, 	\tag{SP}
}
where $\Phi(x, y) := \iprods{y, g(x)}$ is convex in $x$ for $y \in \dom{H^{*}}$ and linear in $y$ for $x\in\dom{F}$.


\paragraph{General convex cone program}
If $H^* (y) := \delta_{\Kc^*}(y)$, the indicator of the dual cone $\Kc^*$ of a proper cone $\Kc$ in $\R^n$, then \eqref{eq:primal_cvx} reduces to
\myeq{eq:cone_cvx_opt}{
	F^{\star} := \min_{x \in \R^p} \Big\{F(x) := f(x) + h(x)\quad \text{s.t.}\quad g(x) \in -\Kc\Big\}.  \tag{CP}
}
This formulation is a general convex cone constrained program, which covers several subclasses such as conic programs and convex programs with nonlinear convex constraints \cite{Bauschke2011}.

Note that our model \eqref{eq:saddle_point} is more general than the linear case $\Phi(x,y) := \iprods{Kx, y}$ in 
\cite{Bauschke2011,Chambolle2011,He2012b,he2016accelerated,Nesterov2005d,Nesterov2005c,sabach2020faster,Shefi2014,TranDinh2015b,tran2019non,tseng2008accelerated}, where $K$ is a given matrix.
Although the convex-linear coupling function $\Phi(x,y) := \iprods{g(x), y}$ is still special compared to general convex-concave ones studied, e.g., in  \cite{juditsky2011first,lin2020near,Nemirovskii2004,thekumparampil2019efficient,hien2017inexact,hamedani2018primal}, the model \eqref{eq:saddle_point} already covers various applications, including those in \cite{juditsky2011first,Nemirovskii2004,hien2017inexact,hamedani2018primal} and requires weaker assumptions.
Let us first review some representative applications and then discuss the limitations of existing works next. 

\paragraph{Representative applications}
If $\Phi(x,y) = \iprods{Kx, y}$, then \eqref{eq:saddle_point} already covers various applications in signal and image processing, compressive sensing, machine learning, and statistics, see, e.g., \cite{Bauschke2011,Chambolle2011,Combettes2011,Esser2010a,Hastie2009}. 
When $\Phi$ is convex-linear, it additionally covers many other key applications in different fields.
For instance, the kernel matrix learning problem for support vector machine studied in \cite[problem (20)]{lanckriet2004learning} can be formulated into \eqref{eq:saddle_point}, where $\Phi$ is quadratic in $x$ (model parameters) and linear in $y$ (a kernel matrix).
Another related problem is the maximum margin clustering studied in \cite{xu2005maximum}, where the coupling objective is also linear in $y$.
Various robust optimization models relying on the well-known Wald's max-min formulation can be cast into \eqref{eq:saddle_point}, where $\Phi$ linearly depends on an uncertainty $y$ as $\Phi(x,y) = \iprods{g(x), y}$, see, e.g., \cite{Ben-Tal2009}.
The generative adversarial networks (GANs) problem involving Wasserstein distances studied in \cite{arjovsky2017wasserstein} can also be formulated as a special case of \eqref{eq:saddle_point} when the discriminator is linear and the generator can be nonlinear w.r.t. their parameters.
This model is also related to optimal transport problems as shown in \cite{farnia2018convex}.
Other applications of \eqref{eq:saddle_point} in machine learning, distributionally robust optimization, game theory, and image and signal processing can be found, e.g., in \cite{Facchinei2003,Kim2008,Peyre2019,rahimian2019distributionally,Shapiro2004}.
It is also worth noting that \eqref{eq:saddle_point} and its special case \eqref{eq:cone_cvx_opt} can serve as subproblems in several nonconvex-concave minimax and nonconvex optimization methods such as proximal-point, inner approximation, and penalty-based schemes, see, e.g., \cite{boob2019proximal,lin2019gradient,TranDinh2012}.

\paragraph{Limitation of existing work}
Methods for solving  \eqref{eq:saddle_point} and its primal problem \eqref{eq:primal_cvx} when $\Phi$ is bilinear are well-developed, see, e.g., \cite{Chambolle2011,combettes2012primal,Davis2014a,Esser2010a,He2012b,Nesterov2005d,Nesterov2005c,TranDinh2015b,tran2019non}.
However, when $\Phi$ is no longer bilinear, algorithms for solving   \eqref{eq:saddle_point} remain limited, see \cite{thekumparampil2019efficient,hamedani2018primal,zhao2019optimal} and their subsequent references.
We find that existing works have the following limitations.

$\diamond$~\textbf{Model assumptions.}
Gradient-based methods such as \cite{lin2020near,hien2017inexact,yang2020catalyst,zhao2019optimal} require $\nabla_x{\Phi}$ and $\nabla_y{\Phi}$ to be uniformly Lipschitz continuous on both $x$ and $y$, which unfortunately excludes some important cases, e.g., the convex cone constrained problem \eqref{eq:cone_cvx_opt}, where $\nabla_x{\Phi}(x,y) = g'(x)^{\top}y$, which is not uniformly Lipschitz continuous on $x$ for all $y$ (see Assumption~\ref{as:A2}). 
In addition, if $H^{*}$ is not strongly convex or restricted strongly convex as in \cite{du2018linear,lin2020near,thekumparampil2019efficient,wang2020improved}, then $H\circ g$ in \eqref{eq:primal_cvx} can be nonsmooth,  which creates several challenges for first-order methods.

$\diamond$~\textbf{High per-iteration complexity.}
Several methods, including \cite{lin2020near,Nemirovskii2004,hien2017inexact,yan2020sharp,yang2020catalyst}, require double loops, where the inner loop approximately solves a subproblem, e.g., the concave maximization problem in $y$, and the outer loop handles the minimization problem in $x$.
These methods can be viewed as an inexact first-order scheme to solve \eqref{eq:primal_cvx}, where the complexity of each outer iteration is often high.
In addition, related parameters such as the inner iteration number are often chosen based on some convergence  bounds, and may depend on a desired accuracy.
This dependence requires sophisticated hyper-parameter tuning strategy to achieve good performance, and it is often challenging to implement in practice. 
For the special case \eqref{eq:cone_cvx_opt}, penalty and augmented Lagrangian approaches (including alternating minimization and alternating direction methods of multipliers (ADMM)) require to solve expensive subproblems \cite{xu2019iteration}. 
There exist very limited single-loop algorithms, e.g., \cite{lin2020near,mokhtari2020unified,mokhtari2020convergence,hamedani2018primal} for general convex-concave minimax problems, and \cite{xu2017first} for a special case of \eqref{eq:cone_cvx_opt}.  
Our method in this paper belongs to a class of single-loop primal-dual algorithms.

$\diamond$~\textbf{Convergence guarantees.}
Subgradient and mirror descent methods such as \cite{juditsky2011first,Nemirovskii2004} often have slow convergence rates compared to gradient and accelerated gradient-based methods \cite{Nesterov2004}.
Hitherto, existing works can only show the best known convergence rates on \textit{ergodic} (or, \textit{averaging}) sequences, via a gap function (\emph{cf.} \eqref{eq:gap_func}), see, e.g., \cite{chen2017accelerated,juditsky2011first,mokhtari2020convergence,Nemirovskii2004,hien2017inexact,Monteiro2010a,Monteiro2011,xu2017first,hamedani2018primal}.
It means that the convergence guarantee is based on an averaging or a  weighted averaging sequence of all the past iterates.
In practical implementation, however, researchers often report performance on the \textit{non-ergodic} (or, the \textit{last-iterate}) sequence, which may only have asymptotic convergence or suboptimal rate compared to the averaging one.
As indicated in \cite{golowich2020last}, the theoretical guarantee on the last-iterate sequence can significantly be slower than the averaging ones.
To achieve faster convergence rates on the last iterates, as shown in \cite{tran2019non,valkonen2019inertial}, one needs to fundamentally redesign the underlying algorithm.
Note that averaging sequences break desired structures of final solutions such as sparsity, low-rankness, or  sharp-edged structure required in many applications, including imaging science.

These three major limitations of existing works motivate us to conduct this research and develop a novel primal-dual algorithm, which affirmatively solves the above challenges.

\paragraph{Our approach}
Problem \eqref{eq:primal_cvx} is much more challenging to solve than its linear case $g(x) = Kx$, especially when the dual domain of \eqref{eq:dual_cvx} is unbounded.
Our approach relies on a novel combination of different techniques.
Firstly, we introduce a surrogate $\Lc$ of the Lagrange function $\widetilde{\Lc}$ in \eqref{eq:saddle_point} and a new potential function $\Lc_{\rho}$ (see Subsect.~\ref{subsec:potential_func}).
Unlike existing potential functions in convex optimization, $\Lc_{\rho}$ is nonconvex in $x$. 
Secondly, we alternatively minimize $\Lc_{\rho}$ w.r.t. to its auxiliary variable $s$ and the primal variable $x$.
The subproblem in $x$ is linearized to use the proximal operator of $h$ and $\nabla{f}$.
Thirdly, we also utilize Nesterov's accelerated momentum step with a new step-size rule to obtain optimal convergence rates.
Finally, we exploit a homotopy strategy developed in \cite{TranDinh2015b,tran2019non} to dynamically update the involved parameters in an explicit manner.
Compared to standard augmented Lagrangian methods, e.g., in \cite{Bertsekas1996a,Rockafellar1976a,xu2017first}, our approach is fundamentally different since it relies on a nonconvex potential function and combines several classical and new techniques in one.

\paragraph{Our contribution}
The contribution of this paper can be summarized as follows:
\begin{itemize}  
\item[$\mathrm{(a)}$] 
We develop a novel single-loop primal-dual algorithm, Algorithm~\ref{alg:A1}, to solve  \eqref{eq:primal_cvx}, \eqref{eq:dual_cvx}, and \eqref{eq:saddle_point} simultaneously, which covers four different variants.
We introduce a new nonconvex potential function to analyze the convergence of our algorithms.

\item[$\mathrm{(b)}$]
We establish the $\BigO{1/k}$ optimal  convergence rates for Algorithm~\ref{alg:A1} in the general convex-linear case on three different criteria:  primal objective residual, dual objective residual, and primal-dual gap, where $k$ is the iteration counter.
Our sublinear convergence rates are achieved via both the averaging sequences and the primal last-iterate sequence, which we call ergodic and semi-ergodic rates, respectively.

\item[$\mathrm{(c)}$]
When $F$ is strongly convex, by deriving new parameter update rules, we establish the $\BigO{1/k^2}$ optimal convergence rates for Algorithm~\ref{alg:A1} on the same three criteria: primal objective residual, dual objective residual, and primal-dual gap.
Again, our convergence rates are achieved via both averaging and primal last-iterate sequences.

\item[$\mathrm{(d)}$]
We specify Algorithm~\ref{alg:A1} to solve the general convex cone constrained program \eqref{eq:cone_cvx_opt}, where our optimal convergence rates, both ergodic and non-ergodic, on the primal objective residual and primal feasibility violation are established.
\end{itemize}  

\paragraph{Comparison with the most related work}
Problem \eqref{eq:saddle_point} can be cast into a special variational inequality (VIP) or a maximally monotone inclusion, where several methods can be applied to solve it, see, e.g., \cite{Bauschke2011,Facchinei2003,juditsky2011solving,kolossoski2017accelerated,Monteiro2010a,Monteiro2011}.
However, the following aspects make our method standout from existing works on the setting \eqref{eq:primal_cvx} and its minimax formulation \eqref{eq:saddle_point}.

Firstly, the main assumption of the VIP approach is the uniformly Lipschitz continuity of the underlying monotone operator, which unfortunately does not hold in our setting (as stated in Assumption~\ref{as:A2}(b,ii)).
Secondly, Algorithm \ref{alg:A1} is different from those in \cite{Bauschke2011,Facchinei2003,Monteiro2010a,Monteiro2011} where we focus on non-asymtopically  sublinear convergence rates under mild assumptions, instead of asymptotic or linear rates as approaches in \cite{Bauschke2011,Facchinei2003,Monteiro2010a,Monteiro2011}.
Thirdly, Algorithm \ref{alg:A1} is single-loop as opposed to double-loop ones as in  \cite{lin2020near,Nemirovskii2004,Nesterov2007a,hien2017inexact,yan2020sharp,yang2020catalyst}.
Note that single-loop algorithms are often easy to implement and extend.
Fourthly, we do not require $\nabla_y{\Phi}$ to be uniformly Lipschitz continuous in $x$ for all $y$ as in \cite{lin2020near,hien2017inexact,yan2020sharp,yang2020catalyst}, where the domain of $y$ in our setting can be unbounded.
Fifthly, compared to other single-loop methods as \cite{mokhtari2020unified,thekumparampil2019efficient,hamedani2018primal}, our rates are non-ergodic on the primal sequence.
To the best of our knowledge, this is the first work establishing such a non-ergodic rates for the non-bilinear case of \eqref{eq:saddle_point}.
Sixthly, we only focus on the general convex case, and the strongly convex case of $F$, and ignore the case when both $F$ and $H^{*}$ are strongly convex since this condition leads to a strong monotonicity of the underlying KKT system of \eqref{eq:saddle_point}, and linear convergence is often well-known \cite{Bauschke2011,Facchinei2003}.
Seventhly, our convergence guarantees are on three different standard criteria, and in a semi-ergodic sense.
Even in the ergodic sense, our parameter updates as well as assumptions are also different from those in \cite{mokhtari2020unified,hamedani2018primal} (see Theorems~\ref{thm:BigO_ergodic_rate} and \ref{thm:alg2_ergodic}).
Finally, our special setting \eqref{eq:cone_cvx_opt} remains more general than the one in \cite{xu2019iteration,xu2017first}.
Our algorithm and its convergence rates stated in Theorem~\ref{thm:conic} are still new compared to \cite{xu2019iteration,xu2017first}.
Our rates include both ergodic and non-ergodic ones as opposed to the ergodic rates in \cite{xu2017first}.


\paragraph{Paper outline}
The rest of this paper is organized as follows.
Section~\ref{sec:background} recalls some basic concepts used in this paper and introduces our new potential function.
Section~\ref{sec:alg1} develops our algorithm, Algorithm~\ref{alg:A1}, for solving \eqref{eq:saddle_point} and its primal and dual formulations.
Its convergence guarantees under different assumptions are also established.
Section~\ref{sec:extension_general_phi} specifies our method to \eqref{eq:cone_cvx_opt}.
Section \ref{sec:num_exp} provides two numerical examples to verify our theoretical results.
For the clarity of presentation, all technical proofs are deferred to the appendices.

\section{Background and New Potential Function}\label{sec:background}
This section recalls some necessary concepts and introduces a new potential function used in this paper.

\subsection{Basic concepts}
We work with Euclidean spaces $\R^p$ and $\R^n$ equipped with a standard inner product $\iprods{u, v}$ and norm $\norm{u}$.
For any nonempty, closed, and convex set $\Xc$ in $\R^p$, $\relint{\Xc}$ denotes its relative interior and $\delta_{\Xc}$ denotes its indicator function.
If $\Kc$ is a convex cone, then $\Kc^{*} := \set{w\in\R^p \mid \iprods{w,x} \geq 0,\ \forall x\in\Kc}$ defines its dual cone.
For any proper, closed, and convex function $f: \R^p \to\Rext$, $\dom{f} := \set{x \in\R^p \mid f(x) < +\infty}$ is its (effective) domain, 
$f^{\ast}(y) := \sup_{x}\{ \iprods{x, y} - f(x) \}$ defines its Fenchel conjugate, $\partial{f}(x) := \{ w\in\R^p \mid f(y) - f(x) \geq \iprods{w, y-x},~\forall y\in\dom{f} \}$ stands for its subdifferential at $x$, 
and $\nabla{f}$ is its gradient or subgradient. 
We also use $\prox_{f}(x) := \mathrm{arg}\min_{y}\{ f(y) + \tfrac{1}{2}\norms{y-x}^2 \}$ to define the proximal operator of $f$.
If $f = \delta_{\Xc}$, then $\prox_f$ reduces to the projection $\proj_{\Xc}$ onto $\Xc$.
For a differentiable vector function $g : \R^p\to\R^n$, $g'(\cdot) \in \R^{n \times p}$ denotes its Jacobian.

A function $f : \R^p \to\R$ is called $M_f$-Lipschitz continuous on $\dom{f}$ with a Lipschitz constant $M_f \in [0, +\infty)$ if $\abss{ f(x) - f(\hat{x})} \leq M_f\norms{x - \hat{x}}$ for all $x, \hat{x} \in\dom{f}$. If $f$ is differentiable on $\dom{f}$ and $\nabla{f}$ is Lipschitz continuous with a Lipschitz constant $L_f \in [0, +\infty)$, i.e., $\norms{\nabla{f}(x) - \nabla{f}(\hat{x})} \leq L_f\norms{x - \hat{x}}$ for $x, \hat{x} \in\dom{f}$, then we say that $f$ is $L_f$-smooth. 
If $f(\cdot) - \frac{\mu_f}{2}\norms{\cdot}^2$ is still convex for some $\mu_f > 0$, then we say that $f$ is $\mu_f$-strongly convex with a strong convexity parameter $\mu_f$. 
If $\mu_f = 0$, then $f$ is just convex.

\subsection{Fundamental assumptions}\label{subsec:opt_cond}
The algorithms developed in this paper rely on the following two assumptions imposed on \eqref{eq:primal_cvx} and \eqref{eq:dual_cvx}.

\begin{assumption}\label{as:A1}
There exists $(x^{\star}, y^{\star})\in \dom{F}\times\dom{H^{*}}$ of \eqref{eq:saddle_point} such that:
\myeq{eq:saddle_point_opt}{
	\widetilde{\Lc}(x^{\star}, y) \leq \widetilde{\Lc}(x^{\star}, y^{\star}) \leq \widetilde{\Lc}(x, y^{\star}),~~\forall (x, y) \in \dom{F}\times\dom{H^{*}}.
}
Moreover, $\dom{F}\times\dom{H^{*}}\subseteq\dom{\Phi}$ with $\Phi(x,y) := \iprods{g(x), y}$ and $\widetilde{\Lc}(x^{\star}, y^{\star})$ is finite.
\end{assumption}
Assumption \ref{as:A1} is standard in convex-concave minimax problems.
It guarantees strong duality $\Pc^{\star} = -\Dc^{\star} = \widetilde{\Lc}(x^{\star}, y^{\star})$ and the existence of solutions for \eqref{eq:primal_cvx} and \eqref{eq:dual_cvx}, see, e.g., \cite{Rockafellar2004}.

\begin{assumption}\label{as:A2}
The functions $F$, $H$, and $\Phi$ in \eqref{eq:primal_cvx} satisfy the following conditions:
\begin{enumerate}
\item[$\mathrm{(a)}$]
	The functions $f$, $h$, and $H$ are proper, closed, and convex on their domain, and $F := f + h$. 
	In addition, $f$ is $L_f$-smooth for some Lipschitz constant $L_f \in [0, \infty)$.

\item[$\mathrm{(b)}$]
	$\Phi(x, y) = \iprods{y, g(x)}$ and is convex in $x$ for any $y \in\dom{H^*}$.
	In addition, it satisfies
	\begin{itemize}
	\item[${(i)}$]
	For any $y \in \dom{H^{*}}$, $g_i(\cdot)$ is $M_{g_i}$-Lipschitz continuous with a Lipschitz constant $M_{g_i} \in [0, +\infty)$ for $i=1,\cdots, n$, i.e. it holds that
	\myeq{eq:Mg_Lipschitz}{
		\abss{g_i(x) - g_i(\hat{x})} \leq  M_{g_i} \norms{x - \hat{x}}, \quad \forall x,\hat{x}\in\dom{g_i}.
	}
	\item[${(ii)}$]
	For any $y\in\dom{H^*}$, $\nabla_x\Phi(\cdot,y) = \iprods{y, g^\prime (\cdot)}$ is $\Lb_g (y)$-Lipschitz continuous with a Lipschitz modulus $\Lb_g (y)\in [0, +\infty)$ depending on $y$, i.e.:
	\myeqn{
		\norms{\nabla_x\Phi(x,y) - \nabla_x\Phi(\hat{x}, y)} \leq \Lb_g (y) \norms{x - \hat{x}},\quad \forall x, \hat{x} \in\dom{g}.
	}
	Moreover, $0 \leq \Lb_g (y) \leq L_g \norms{y}$ for a uniform constant $L_g \in [0, +\infty)$.
	\end{itemize}
\end{enumerate}
\end{assumption}
Assumption~\ref{as:A2} is relatively technical, but widely used in the literature when developing first-order primal-dual methods for \eqref{eq:primal_cvx}, see, e.g., \cite{lin2020near,hien2017inexact,yan2020sharp,yang2020catalyst,hamedani2018primal}. 
However, unlike these works, $\Lb_g (y)$ in  Assumption \ref{as:A2} depends on $y$ such that $\Lb_g(y) \leq L_b\norms{y}$, which allows us to cover, e.g., the convex cone constrained problem \eqref{eq:cone_cvx_opt} without requiring the boundedness of $\dom{H^{*}}$. 
Since $\nabla_{y_i}\Phi(x,y) = g_i(x)$, \eqref{eq:Mg_Lipschitz} is equivalent to the $M_{g_i}$-Lipschitz continuity of $g_i$ for $i = 1, \cdots, n$.
Hence, we have
\vspace{-1ex}
\myeq{eq:M_g}{
\hspace{-1ex}
\norms{\nabla_y\Phi(x,y) - \nabla_y\Phi(\hat{x},y)}^2 = \norms{g(x) - g(\hat{x})}^2 \leq M_g^2\norms{x - \hat{x}}^2, \quad \text{where}\quad M_g^2 := \sum_{i=1}^nM_{g_i}^2.
\hspace{-1ex}
\vspace{-1ex}
}
Clearly, if $\Phi(x,y) = \iprods{Kx, y}$ is bilinear, then it automatically satisfies Assumptions~\ref{as:A2}.

\subsection{Optimality condition and gap function}\label{subsec:gap}
In view of Assumption \ref{as:A1}, there exists a saddle-point $(x^{\star}, y^{\star})\in\dom{F}\times\dom{H^{*}}$ of \eqref{eq:saddle_point} that satisfies
\myeq{eq:opt_cond}{
	0 \in \partial F(x^{\star}) +  g'(x^{\star})^\top y^{\star}\quad \text{and}\quad 0 \in g(x^{\star}) - \partial H^{*} (y^{\star}).
}
To characterize saddle-points of \eqref{eq:saddle_point}, we define the following gap function, see \cite{Chambolle2011,Facchinei2003,Nemirovskii2004}:
\myeq{eq:gap_func}{
	\Gc_{\Xc\times\Yc}(x, y) := \sup\big\{\widetilde{\Lc}(x, \hat{y}) - \widetilde{\Lc}(\hat{x}, y) : \hat{x} \in\Xc, \ \hat{y} \in\Yc \big\}, 
}
where $\Xc \subseteq \dom{F}$ and $\Yc \subseteq \dom{H^{*}}$ are two nonempty and closed subsets such that $\Xc \times \Yc$ contains a saddle-point $(x^{\star}, y^{\star})$. 
It is clear that $\Gc_{\Xc\times\Yc}(x, y) \geq 0$ for any $(x, y) \in \dom{F} \times \dom{H^{*}}$.
If $(x^{\star}, y^{\star})$ is a saddle-point of \eqref{eq:saddle_point}, then $\Gc_{\Xc \times \Yc} (x^{\star}, y^{\star}) = 0$.

\subsection{New potential function and its key property}\label{subsec:potential_func}
One of the main tools to develop our algorithm is a novel potential function, which is formalized as follows.
First, we upper bound $H^{*}$ using its Fenchel conjugate, i.e., $H^{*}(y) \geq \iprods{s, y} - H(s)$ for any $s\in\dom{H}$.
Consequently, we can upper bound the Lagrange function $\widetilde{\Lc}$ of \eqref{eq:saddle_point} by 
\myeq{eq:Lag_w_s}{
	\Lc(x, s, y) := F(x) + \Phi(x, y) - \iprods{s, y} + H(s) =  F(x) + H(s) + \iprods{y,\ g(x) - s}.
}
Clearly, for any $(x, s, y) \in \dom{F}\times\dom{H}\times\dom{H^{*}}$, we have
\myeq{eq:L_to_Ltilde}{
	\widetilde{\Lc}(x, y) \leq \Lc(x, s, y)\quad \text{and}\quad \widetilde{\Lc}(x, y) = \Lc(x, s, y)~\text{iff}~s\in \partial H^* (y).
}
As a result, for any $(x, s, y) \in \dom{F} \times \dom{H} \times \dom{H^{*}}$ and $s^{\star} \in\partial{H^{*}}(y^{\star})$, \eqref{eq:saddle_point_opt} implies {\!\!\!}
\myeq{eq:saddle_pt_s}{
	\widetilde{\Lc} (x^{\star}, y) \leq \Lc(x^{\star}, s^{\star}, y) = \widetilde{\Lc} (x^{\star}, y^{\star}) = \Lc(x^{\star}, s^{\star}, y^{\star}) \leq \widetilde{\Lc}(x, y^{\star})\leq \Lc(x, s, y^{\star}).
}
Next, let us introduce a \textbf{new potential function} as follows:
\myeq{eq:aug_Lag}{
	\Lc_{\rho}  (x, s, y) := \Lc(x, s, y) + \frac{\rho}{2}\norms{g(x) - s}^2 \equiv  F(x) + H(s) + \phi_{\rho}  (x, s, y),
}
where $\rho > 0$ is a given parameter and 
\myeq{eq:phi_func}{
	\phi_{\rho}  (x, s, y) := \iprods{y,\ g(x) - s} +  \frac{\rho}{2} \norms{g(x) - s}^2.
}
%
Unlike existing potential functions used in convex optimization, if $g$ is not affine, then our potential function $\Lc_\rho$ in \eqref{eq:aug_Lag} is generally \textbf{nonconvex} in $x$.
%

\begin{lemma}\label{lm:Lipschitz_continuous}
Let $\phi_{\rho}$ be defined by \eqref{eq:phi_func} and Assumption~\ref{as:A2} holds. 
Then, $\nabla_x \phi_{\rho}  (x, s, y)  =  {[g^{\prime} (x)]}^\top \big(y + \rho[g(x) - s]\big)$ and $\nabla_s \phi_{\rho}  (x, s, y)  =  -y - \rho[g(x) - s]$.

For any $x, \hat{x} \in \dom{g}$ and $y, s, \hat{s} \in \R^n$  such that $y + \rho [g(x) - s] \in \dom{H^{*}}$, we define
\myeq{eq:Delta_rho}{
\hspace{-1ex}
	\Delta_{\rho}(\hat{x}, \hat{s}; x,s, y) := \phi_{\rho}  (\hat{x}, \hat{s}, y) \! - \!  \phi_{\rho}  (x, s, y) \! - \! \iprods{\nabla_x \phi_{\rho}  (x, s, y), \hat{x} \! - \! x} - \iprods{\nabla_s \phi_{\rho}  (x, s, y), \hat{s} \! -\!  s}.
\hspace{-3ex}	
}
Then, the following estimate holds:
\myeq{eq:pd_Lipschitz_local_onlynl}{
	0 \leq \Delta_{\rho}(\hat{x}, \hat{s}; x,s, y) - \frac{\rho}{2} \norms{[g(\hat{x}) -  \hat{s}] - [g(x) - s]}^2 \leq \frac{1}{2}\Lb_g \left(y + \rho[g(x) - s]\right) \norms{\hat{x} - x}^2,
}
where $\Lb_g\left(y + \rho[g(x) - s]\right)$ is the Lipschitz modulus defined in Item $($ii$)$ of Assumption~\ref{as:A2}$\mathrm{(b)}$.
\end{lemma}

\begin{proof}
From \eqref{eq:phi_func}, one can directly compute $\nabla_x \phi_{\rho}  (x, s, y)  =  {[g^{\prime} (x)]}^\top \big(y + \rho[g(x) - s]\big)$ and $\nabla_s \phi_{\rho}  (x, s, y)  =  -y - \rho[g(x) - s]$ of $\phi_{\rho}$ as shown in the lemma.

Now, by the definition of $\phi_{\rho} (x, s, y)$ in \eqref{eq:phi_func} and of $\Delta_\rho$ in \eqref{eq:Delta_rho}, and the partial gradients w.r.t. $x$ and $s$ of $\phi_{\rho}$, we can explicitly write $\Delta_\rho$ as
\myeq{eq:lm_c_proof_g}{
\hspace{-6ex}
\arraycolsep=0.1em
\begin{array}{lcl}
	\Delta_{\rho}(\hat{x}, \hat{s}; x, s, y)  & = & \iprods{y,\ [g(\hat{x}) - \hat{s}] - [g(x) - s]} + \frac{\rho}{2} \left[ \norms{g(\hat{x}) - \hat{s}}^2 - \norms{g(x) - s}^2\right] \vspace{1ex}\\
	& & - {~} \iprod{y + \rho [g (x) - s],\ g^{\prime} (x)(\hat{x} - x) - (\hat{s} - s)}\vspace{1ex}\\
	& = & \iprods{y + \rho[g(x) - s],\ g(\hat{x}) - g(x) - g^{\prime} (x)(\hat{x} - x)}  +  \frac{\rho}{2} \norms{[g(\hat{x}) - \hat{s}] - [g(x) - s]}^2.
\end{array}
\hspace{-4ex}
}
By the $\Lb_g(\cdot)$-smoothness of $\nabla_x{\Phi}(x, \cdot)$ w.r.t. $x$ and $y + \rho[g(x) - s] \in \dom{H^{*}}$, we have
\myeq{eq:lm_c_proof_g_y}{
	0 \leq \iprod{y + \rho[g (x) - s],\ g(\hat{x}) - g(x) - g^{\prime} (x)(\hat{x} - x)} \leq \tfrac{1}{2}\Lb_g \left(y + \rho[g(x) - s]\right) \norms{\hat{x} - x}^2.
}
Combining \eqref{eq:lm_c_proof_g} and \eqref{eq:lm_c_proof_g_y}, we immediately get  \eqref{eq:pd_Lipschitz_local_onlynl}.
\end{proof}

\section{New Primal-Dual Algorithm and Convergence Guarantees}\label{sec:alg1}
In this section, we  develop a novel primal-dual algorithm to solve \eqref{eq:primal_cvx} and its dual form \eqref{eq:dual_cvx}. 

\subsection{Intuition and derivation}\label{subsec:derivation_of_pd_scheme1}
Our main idea is to exploit the potential function $\Lc_\rho$ defined in \eqref{eq:aug_Lag} to measure the progress of an iterate sequence $\set{(x^k, y^k)}$ generated by the proposed algorithm.
Since this function not only involves $x$ but also the dual variable $y$ and the auxiliary variable $s$, we also need to update them in an alternating manner. 

More specifically, at each iteration $k\geq 0$, given $\hat{x}^k \in \dom{F}$ and $\tilde{y}^k \in \dom{H^{*}}$, we can derive the main step of our algorithm as follows:
\begin{enumerate}
\item[\textbf{Step 1:}] 
We first update the auxiliary variable $s^{k + 1}$ by minimizing $\Lc_{\rho_k} (\hat{x}^k, s, \tilde{y}^k)$ w.r.t. $s$:
\myeq{eq:scheme_s}{
\arraycolsep=0.2em
\begin{array}{lrl}
s^{k + 1} &:=& \displaystyle\argmin_{s} \set{H(s) + \phi_{\rho_k}(\hat{x}^k, s, \tilde{y}^k) }  =  \prox_{H/\rho_k} \Big(\tfrac{\tilde{y}^k}{\rho_k} + g(\hat{x}^k) \Big).
\end{array}
}
\item[\textbf{Step 2:}]
Given $s^{k+1}$ computed by \eqref{eq:scheme_s}, we minimize $\Lc_{\rho_k} (x, s^{k + 1}, \tilde{y}^k)$ w.r.t. $x$ to obtain $x^{k+1}$. 
However, minimizing this function directly is difficult.
We instead linearize $f(\cdot) + \phi_{\rho_k} (\cdot, s^{k + 1}, \tilde{y}^k)$ at $\hat{x}^k$ and then minimize the surrogate of $\Lc_{\rho_k}$ as
\myeq{eq:scheme_x}{
\hspace{-4ex}
\arraycolsep=0.2em
\begin{array}{lrl}
	x^{k + 1} & := & \argmin_{x} \left\{ h(x) + \iprods{\nabla{f}(\hat{x}^k) + \nabla_x \phi_{\rho_k} (\hat{x}^k, s^{k + 1}, \tilde{y}^k), x - \hat{x}^k} + \tfrac{L_k}{2} \norms{x - \hat{x}^k}^2 \right\}\vspace{1ex}\\
	& = & \prox_{h/L_k} \Big(\hat{x}^k - \frac{1}{L_k}\left[\nabla{f}(\hat{x}^k) +  \nabla_x \phi_{\rho_k} (\hat{x}^k, s^{k + 1}, \tilde{y}^k)\right] \Big),
\end{array}
\hspace{-4ex}	
}
where $L_k > 0$ is a given parameter, which will be determined later.

\item[\textbf{Step 3:}] 
Next, based on our analysis in Lemma~\ref{lm:descent_property1}, we update $\tilde{y}^k$ with a  step-size $\eta_k > 0$ as
\myeq{eq:scheme_ytilde}{
	\tilde{y}^{k + 1} :=   \tilde{y}^k + \eta_k \big([g(x^{k + 1}) - s^{k + 1}] - (1 - \tau_k)[g(x^k) - s^k]\big).
}

\item[\textbf{Step 4:}] Finally, to obtain a dual convergence, given $\breve{y}^k$ and $\tau_k\in (0, 1]$, we update $\breve{y}^{k+1}$ as
\myeq{eq:y_average}{
	\breve{y}^{k+1} := (1-\tau_k)\breve{y}^k + \tau_k\big( \tilde{y}^k + \rho_k[g(\hat{x}^k) - s^{k+1}]\big). 
}
\end{enumerate}
While \textbf{Steps 1} to \textbf{4} are the main components of our algorithm, the update of $\hat{x}^k$ and the parameters is also an important part of our method, which will be specified in the sequel.

\subsection{One-iteration analysis}\label{subsec:one_iterations}
Our first key result is the following lemma, which establishes a recursive inequality of $(s^{k+1}, x^{k+1}, \tilde{y}^{k+1}, \breve{y}^{k+1})$ generated by \textbf{Steps 1} to \textbf{4} in Subsection~\ref{subsec:derivation_of_pd_scheme1}.
For the sake of presentation, we defer its proof to Appendix \ref{subsec:A_alg1_keylm}.

\begin{lemma}\label{lm:descent_property1}
Suppose that Assumptions~\ref{as:A1} and \ref{as:A2} hold.
Let $s^{k+1}$, $x^{k+1}$, $\tilde{y}^{k+1}$, and $\breve{y}^{k+1}$ be computed by  \eqref{eq:scheme_s}, \eqref{eq:scheme_x}, \eqref{eq:scheme_ytilde}, and \eqref{eq:y_average}, respectively.
Let  $\Lc$ be given in \eqref{eq:Lag_w_s} and $\Lc_\rho$ be given in \eqref{eq:aug_Lag}.
Then, for $\tau_k \in (0, 1]$ and $(x, s, y) \in \dom{F}\times\dom{H}\times\dom{H^{*}}$, we have
\myeq{eq:p2_keylemma}{
\hspace{-1ex}
\arraycolsep=0.3em
\begin{array}{rl}
	\Lc_{\rho_k} (x^{k + 1}, s^{k + 1}, y) & - {~} \Lc(x, s, \breve{y}^{k + 1}) \ \leq \ (1 - \tau_k)\big[ \Lc_{\rho_{k - 1}} (x^k, s^k, y) - \Lc(x, s, \breve{y}^k) \big]\vspace{1ex}\\
		& + {~} \frac{\tau_k^2}{2} \big(L_k - \mu_f\big) \norms{\frac{1}{\tau_k}[\hat{x}^k - (1-\tau_k)x^k] - x}^2 \vspace{1ex}\\
		& -  {~} \frac{\tau_k^2}{2}\big(L_k + \mu_h\big) \norms{\frac{1}{\tau_k}[x^{k + 1} - (1-\tau_k)x^k] - x}^2 \vspace{1ex}\\
		& - {~} \frac{(\mu_f+\mu_h)\tau_k(1-\tau_k)}{2}\norms{x^k - x}^2 + \frac{1}{2\eta_k} \big[\norms{\tilde{y}^k - y}^2 - \norms{\tilde{y}^{k + 1} - y}^2 \big] \vspace{1ex}\\
		& - {~} \frac{(1 - \tau_k)}{2} \big[\rho_{k - 1} - (1 - \tau_k)\rho_k\big] \norms{g(x^k) - s^k}^2\vspace{1ex}\\
		& - {~} \frac12 \left(L_k -  \Lb_g^k - L_f - \frac{\rho_k^2 M_g^2}{\rho_k - \eta_k}\right)\norms{x^{k + 1} - \hat{x}^k}^2, 
\end{array}
\hspace{-1ex}
}
where $\Lb_g^k :=  \Lb_g\big( \tilde{y}^k + \rho_k [g(\hat{x}^k) - s^{k + 1}] \big)$, $L_f$ is defined in Assumption~\ref{as:A2}, $L_k > 0$, $\rho_k > \eta_k > 0$, and $\mu_f \geq 0$ and $\mu_h \geq 0$ are the strong convexity parameters of $f$ and $g$, respectively.
\end{lemma}

\subsection{The full algorithm}\label{subsec:full_algorithm}
To write our algorithm into a primal-dual form using $\prox_{h/L_k}(\cdot)$ and $\prox_{\rho_kH^{*}}(\cdot)$ as in, e.g., \cite{Chambolle2011,Esser2010a,tran2019non,hamedani2018primal}, we define a new variable
\myeq{eq:in_domH}{
y^{k + 1} := \prox_{\rho_k H^{*}} \left(\tilde{y}^k + \rho_k g(\hat{x}^k)\right).
}
Then, by Moreau's identity \cite[Theorem 14.3]{Bauschke2011}, we can show that
\myeq{eq:keylm_yk_in_domH}{
	\rho_k s^{k + 1} \stackrel{\eqref{eq:scheme_s}}=  [\tilde{y}^k\!+\!\rho_k g(\hat{x}^k)] -\prox_{\rho_k H^{*}} \left(\tilde{y}^k + \rho_k g(\hat{x}^k)\right) = [\tilde{y}^k\!+\!\rho_k g(\hat{x}^k)] - y^{k + 1}.
}
Thus we can in fact eliminate variable $s^{k + 1}$ from the expression of $x^{k + 1}$ in \eqref{eq:scheme_x} by noting that $\nabla_x \phi_{\rho_k} (\hat{x}^k, s^{k + 1}, \tilde{y}^k) = g^\prime (\hat{x}^k)^{\top} y^{k + 1}$.
Similarly, the presence of $s^k$ and $s^{k + 1}$ in the update of $\tilde{y}^{k + 1}$ in \eqref{eq:scheme_ytilde} can also be eliminated by using \eqref{eq:keylm_yk_in_domH}.

Now, to update $\hat{x}^k$, we follow Nesterov's accelerated scheme \cite{Nesterov2004} as $\hat{x}^{k+1} := x^{k+1} + \beta_{k+1}(x^{k+1} - x^k)$ for some coefficient $\beta_{k+1} \geq 0$.
We emphasize that the update of $\beta_{k+1}$ in Theorems \ref{thm:alg2_ergodic} and \ref{thm:BigO_nonegrodic_rate_str} below is new compared to existing methods even in the bilinear case.

Finally, putting all the above steps together, we obtain a complete single-loop primal-dual first-order algorithm as specified in Algorithm \ref{alg:A1}.

\begin{algorithm}[ht!]
\caption{(Single-Loop Primal-Dual First-Order Algorithm)}\label{alg:A1}
\begin{algorithmic}[1]
\normalsize
	\State {\bfseries Initialization:}
		Choose an initial point $(x^0, y^0) \in \dom{F}\times\dom{H^{*}}$.
	\State\hspace{2ex}\label{step:A1:init_var}
		Set $\hat{x}^0 := x^0$,  $\tilde{y}^0 := y^0$, and $\Theta_0 := 0$. 
		Set $\breve{y}^0 := y^0$ for Theorem~\ref{thm:BigO_nonegrodic_rate} or \ref{thm:BigO_nonegrodic_rate_str}.
   	\State\hspace{2ex}\label{step:A1:init}
		Choose $\tau_0$, $L_0$, $\rho_0$, and $\eta_0$ according to Theorem~\ref{thm:BigO_ergodic_rate}, \ref{thm:alg2_ergodic}, \ref{thm:BigO_nonegrodic_rate}, or \ref{thm:BigO_nonegrodic_rate_str}.
   	\State\textbf{For~}{$k := 0$ {\bfseries to} $k_{\max}$}
   	\State\hspace{2ex}\label{step:A1:main_step1}
		Update $\tau_k$, $L_k$, $\rho_k$, $\eta_k$, and $\beta_k$ as in Theorem~\ref{thm:BigO_ergodic_rate}, \ref{thm:alg2_ergodic}, \ref{thm:BigO_nonegrodic_rate}, or \ref{thm:BigO_nonegrodic_rate_str} consistently with Step~\ref{step:A1:init}.
		\vspace{0.5ex}
	\State\hspace{2ex}\label{step:A1:s1}
		Update
   	\vspace{-1ex}
	\myeq{eq:APD_scheme1}{
	\left\{\arraycolsep=0.2em
	\begin{array}{lcl}
		y^{k + 1} & := & \prox_{\rho_k H^*} \left(\tilde{y}^k + \rho_k g(\hat{x}^k)\right),\vspace{1ex}\\
		x^{k + 1} & := & \prox_{h/L_k} \left(\hat{x}^k - \frac{1}{L_k} [\nabla f(\hat{x}^k) + g'(\hat{x}^k)^\top y^{k + 1}]\right),\vspace{1ex}\\
		\Theta_{k+1}  & := & g(x^{k + 1}) - g(\hat{x}^k) + \frac{1}{\rho_k} (y^{k + 1} - \tilde{y}^k), \vspace{1ex}\\
		\tilde{y}^{k+1} & := &  \tilde{y}^k + \eta_k\left[\Theta_{k+1} - (1-\tau_k)\Theta_k\right], \vspace{1ex}\\
		\hat{x}^{k + 1} & := & x^{k + 1} + \beta_{k+1} (x^{k + 1} - x^k).
	\end{array}\right.
	\vspace{-0.5ex}
   	}
	\State\hspace{2ex}\label{step:A1:main_step3}
	Update $\breve{y}^{k+1} := (1-\tau_k)\breve{y}^k + \tau_ky^{k+1}$ for the variants in Theorems~\ref{thm:BigO_nonegrodic_rate} and \ref{thm:BigO_nonegrodic_rate_str}.
   	\Statex\textbf{EndFor}
\end{algorithmic}
\end{algorithm}

\paragraph{Per-iteration complexity}
The main computation of Algorithm~\ref{alg:A1} is \eqref{eq:APD_scheme1} of Step \ref{step:A1:s1}, where
\begin{enumerate}
	\item Line 1 requires one function evaluation $g(\hat{x}^k)$, which is exactly $\nabla_y\phi_{\rho_k}(\hat{x}^k,\tilde{y}^k)$, and one $\prox_{\rho_kH^{*}}(\cdot)$ operation;
	\item Line 2 needs one Jacobian-vector product $g'(\hat{x}^k)^{\top}y^{k+1}$, which is $\nabla_x{\phi_{\rho_k}}(\hat{x}^k, y^{k+1})$, one gradient $\nabla{f}(\hat{x}^k)$, and one $\prox_{h/L_k}(\cdot)$ operation;
	\item Line 3 needs one more function evaluation $g(x^{k + 1})$ if $\hat{x}^{k+1} \neq x^{k+1}$;
	\item Lines 4 and 5 use only vector additions and scalar-vector multiplications.
\end{enumerate}
This break-down of complexity shows that Algorithm \ref{alg:A1} essentially has  the same complexity as other state-of-the-art single-loop primal-dual first-order algorithms such as \cite{mokhtari2020unified,hamedani2018primal}. 

\subsection{Ergodic convergence rates}\label{subsec:alg1_thms}
Let us first analyze the ergodic  convergence rates of Algorithm~\ref{alg:A1} by setting $\tau_k = 1$ and updating other parameters accordingly for all iterations.

\subsubsection{General convex case}
When $F$ in \eqref{eq:primal_cvx} is only convex (i.e., $\mu_F = 0$), the following theorem establishes the $\BigO{1/k}$ ergodic convergence rate of Algorithm~\ref{alg:A1}, whose proof is postponed to Appendix \ref{subsec:alg1_thm_ergodic}.

\begin{theorem}\label{thm:BigO_ergodic_rate}
Suppose that Assumptions \ref{as:A1} and \ref{as:A2} hold for \eqref{eq:primal_cvx}. 
Let $\set{(x^k, y^k)}$ be generated by Algorithm \ref{alg:A1} using the following parameter update rules:
\begin{itemize}
\item\textbf{Initialization:} 
For a given $(x^0, y^0)\in\dom{F}\times\dom{H^{*}}$ and a saddle-point $(x^{\star}, y^{\star})$, let $D$ be an arbitrary upper bound constant such that $D \geq \max\sets{\norms{x^0 - x^{\star}}, \norms{y^0 - y^{\star}}, \norms{y^{\star}}}$.
Define
\myeq{eq:non_erg_T_def2}{
\left\{\begin{array}{l}
	C := \max\set{L_f + 2M_g^2 + 2, \ L_g D(L_g D + 4M_g + 2)},\vspace{1ex}\\
	\rho := 1, \qquad \eta := \frac{\rho}{2}, \quad \text{and} \quad L := L_f + \rho( C + 2M_g^2).
	\end{array}\right.
}
\item \textbf{Parameter update:} For all $k \geq 0$, we fix the parameters of Algorithm~\ref{alg:A1} at
\myeq{eq:param_update_ergodic_O1}{
	\tau_k := 1,\quad \beta_{k+1} := 0, \quad \rho_k :=  \rho,\quad L_k := L,\quad\text{and}\quad \eta_k := \eta.
}
\end{itemize}
Let $(\bar{x}^k,\ \bar{y}^k)$ be computed as $(\bar{x}^k,\ \bar{y}^k) := \frac{1}{k}\sum_{j=1}^k(x^j,\ y^j)$.
Then, for all $k \geq 1$, we have  
\myeq{eq:gap_convergence_ergodic}{
\left\{\arraycolsep=0.3em
\begin{array}{lcl}
	\Pc(\bar{x}^k) - \Pc^{\star} & \leq & \dfrac{1}{2k} \Big[ L_0 \norms{x^0 - x^{\star}}^2  + \frac{2}{\rho_0}\big(\norms{y^0} + M_H \big)^2 \Big],\vspace{1ex}\\
	\Dc(\bar{y}^k) - \Dc^{\star} & \leq & \dfrac{1}{2k} \Big[ L_0 \big(\norms{x^0} + M_{F^*}\big)^2  + \frac{2}{\rho_0}\norms{y^0 - y^{\star}}^2 \Big], \vspace{1ex}\\	
	\Gc_{\Xc\times\Yc}(\bar{x}^k,\bar{y}^k) & \leq & \dfrac{\Rc^2_{\Xc\times\Yc}}{2k}, 
\end{array}\right.
}
where  $\Gc_{\Xc \times \Yc}$ is defined by \eqref{eq:gap_func}, $\Rc^2_{\Xc\times\Yc} := \sup\big\{ L_0\norms{x - x^0}^2 + \frac{2}{\rho_0}\norms{y - y^0}^2 : x\in\Xc, y \in \Yc \big\}$, and $M_H,\ M_{F^{*}} \in [0, +\infty]$ are the Lipschitz constants of $H$ and $F^*$, respectively.
\end{theorem}

The constant $C$ in \eqref{eq:non_erg_T_def2} only requires an upper bound $D$.
In many cases, we can estimate this quantity using properties of the functions and/or constraints of \eqref{eq:primal_cvx}.
The finer $D$ is estimated, the better performance of Algorithm~\ref{alg:A1} we get.
Note that we can also choose an arbitrary $\rho > 0$ and $0 < \eta < \rho$, but we will need to recompute $C$ in \eqref{eq:non_erg_T_def2} to satisfy \eqref{eq:para_cond1}.

\subsubsection{Strongly convex case}
Now, we consider the case where $F$ in \eqref{eq:primal_cvx} is $\mu_F$-strongly convex with $\mu_F > 0$.
The following theorem establishes the $\BigO{1/k^2}$ ergodic convergence rates of Algorithm~\ref{alg:A1}, whose proof is given in Appendix \ref{subsec:alg2_ergodic_proof}.

\begin{theorem}\label{thm:alg2_ergodic}
Suppose that Assumptions \ref{as:A1} and \ref{as:A2} hold for \eqref{eq:primal_cvx}.  
Suppose further that $F$ in \eqref{eq:primal_cvx} is $\mu_F$-strongly convex with $\mu_F > 0$.
Let $\sets{(x^k, y^k)}$ be generated by Algorithm \ref{alg:A1} using the following parameter update rules:
\begin{itemize}
\item \textbf{Initialization:} 
Let $C$ be the constant computed as in \eqref{eq:non_erg_T_def2}.
Given $\rho_0 := 1$, let
\myeq{eq:scvx_ergodic_param_init}{
\left\{\begin{array}{l}
	L_0 := L_f + \rho_0(C + 2M_g^2), \vspace{1ex}\\
	P_0 := \frac{\rho_0}{2L_0}\big[ \sqrt{4L_0(\mu_f + \mu_h) + (2L_0 - \mu_f)^2} - (2L_0 - \mu_f) \big] > 0. 
\end{array}\right.
}
\item \textbf{Parameter update:} For all $k \geq 0$, we  update
\myeq{eq:scvx_ergodic_param_update}{
\left\{\begin{array}{l}
\tau_k := 1,\quad \beta_{k+1} := 0,\quad \theta_{k+1} := \frac{2L_k}{\mu_f + \sqrt{\mu_f^2 + 4L_k(L_k + \mu_h)}}, \vspace{1ex}\\
L_{k + 1} := \frac{L_k}{\theta_{k + 1}}, \quad \rho_{k + 1} := \frac{\rho_k}{\theta_{k + 1}}, \quad\text{and} \quad  \eta_{k+1} := \frac{1}{2}\rho_{k+1}.
\end{array}\right.
}
\end{itemize}
Let $\sets{(\bar{x}^k, \bar{y}^k)}$ be an ergodic sequence defined as
\myeq{eq:ergodic_O2_xbar_sbar}{
	(\bar{x}^k, \bar{y}^k) := \frac{1}{\Sigma_k} \sum_{j = 0}^{k - 1} \rho_j (x^{j+1}, y^{j+1}),\quad\text{where}\quad \Sigma_k := \sum_{j = 0}^{k - 1} \rho_j.
}
Then, for any $k \geq 1$, the following bounds hold:
\myeq{eq:ergodic_O2_result}{
\arraycolsep=0.2em
\left\{\begin{array}{rlcl}
	& \Pc(\bar{x}^k) - \Pc^{\star} & \leq & \dfrac{1}{2\rho_0k + P_0k(k - 1)} \Big[ L_0 \norms{x^0 - x^{\star} }^2  + \frac{2}{\rho_0} \big(\norms{y^0} + M_H \big)^2 \Big],\vspace{1ex}\\
	& \Dc(\bar{y}^k) - \Dc^{\star} & \leq & \dfrac{1}{2\rho_0k + P_0k(k - 1)} \Big[ L_0 \big(\norms{x^0} + M_{F^{*}} \big)^2 +  \frac{2}{\rho_0}\norms{y^0 - y^{\star} }^2 \Big], \vspace{1ex}\\
	& \Gc_{\Xc \times \Yc} (\bar{x}^k, \bar{y}^k) & \leq & \dfrac{\Rc^2_{\Xc\times\Yc}}{2\rho_0k + P_0k(k - 1)}, 
\end{array}\right.
}
where  $\Gc_{\Xc \times \Yc}$ is defined by \eqref{eq:gap_func}, $\Rc_{\Xc\times\Yc}^2$ is defined in Theorem~\ref{thm:BigO_ergodic_rate}, and $M_H,\ M_{F^*} \in [0, +\infty]$ are the Lipschitz constants of $H$ and $F^*$, respectively.
\end{theorem}

\begin{remark}\label{re:zero_mu}
If $\mu_F = 0$, then $\theta_{k+1} = 1$ in \eqref{eq:scvx_ergodic_param_update}, and consequently, $L_k = L_0$, $\rho_k = \rho_0$, and $\eta_k = \eta_0$ for all $k \geq 0$.
Moreover, $P_0 = 0$ in \eqref{eq:scvx_ergodic_param_init}.
Hence, the convergence rates in \eqref{eq:ergodic_O2_result} reduce to $\BigO{1/k}$ as in Theorem~\ref{thm:BigO_ergodic_rate}.
Therefore, to achieve $\BigO{1/k^2}$ accelerated rates, we require either $\mu_f > 0$ or $\mu_h > 0$, i.e., either $f$ or $h$ in \eqref{eq:primal_cvx} is strongly convex.
\end{remark}

With the update rules \eqref{eq:non_erg_T_def2} and \eqref{eq:param_update_ergodic_O1} (or, respectively, \eqref{eq:scvx_ergodic_param_init} and \eqref{eq:scvx_ergodic_param_update}), the main step \eqref{eq:APD_scheme1} of Algorithm~\ref{alg:A1} can be simplified as 
\myeq{eq:simplified_pd}{
\left\{\begin{array}{lcl}
	y^{k + 1} & := & \prox_{\rho_k H^*} \left(\tilde{y}^k + \rho_kg(x^k)\right),\vspace{1ex}\\
	x^{k + 1} & := & \prox_{h/L_k} \Big(x^k - \frac{1}{L_k}\big[\nabla f(x^k) + g'(x^k)^{\top}y^{k + 1} \big] \Big),\vspace{1ex}\\
	\tilde{y}^{k+1} & := &  \tilde{y}^k + \eta_k\big[ g(x^{k + 1}) - g(x^k) + \frac{1}{\rho_k} (y^{k + 1} - \tilde{y}^k)\big].
\end{array}\right.
}
Clearly, \eqref{eq:simplified_pd} requires one $\prox_{\rho_k H^{*}}(\cdot)$ operation, one $\prox_{h/L_k}(\cdot)$ operation, one function value $g(x^k)$, one gradient $\nabla f(x^k)$, and one Jacobian-vector product $g^{\prime}(x^k)^{\top}y^{k+1}$.
Note that \eqref{eq:simplified_pd} is still different from \cite[Algorithm 1]{hamedani2018primal} when solving \eqref{eq:saddle_point}.
If $H = \delta_{\R^m_{+}\times\sets{0}^n}$, the indicator of $\R^m_{+}\times\sets{0}^n$, then \eqref{eq:simplified_pd} reduces to the one that  is similar to \cite[Algorithm 1]{xu2017first} for solving a special case of \eqref{eq:cone_cvx_opt}, but our step $\tilde{y}^k$ is different from  \cite{xu2017first}.

\subsection{Semi-ergodic convergence rates}
In this subsection, we analyze the semi-ergodic convergence rates (i.e., the rates on the primal last iterate sequence and the dual averaging sequence) of Algorithm~\ref{alg:A1} for two cases: general convexity and strong convexity.

\subsubsection{The general convex case}
Assume that $F$ and $H^{*}$ are only convex, but not necessarily strongly convex, and $h$ and $H^{*}$ are not necessarily smooth.
Theorem~\ref{thm:BigO_nonegrodic_rate} (see Appendix \ref{subsec:alg1_thm_non_ergodic} for its proof) shows the $\BigO{1/k}$ semi-ergodic optimal rate of Algorithm \ref{alg:A1} on the primal last-iterate sequence $\sets{x^k}$ and the dual averaging sequence $\sets{\breve{y}^k}$.

\begin{theorem}\label{thm:BigO_nonegrodic_rate}
Suppose that Assumptions \ref{as:A1} and \ref{as:A2} hold for \eqref{eq:primal_cvx}.
Assume, in addition, that $H$ is $M_H$-Lipschitz continuous with a Lipschitz constant $M_H$ such that $L_gM_H <+ \infty$.
In particular, if $g$ is affine, then $L_g = 0$, and we allow $M_H = +\infty$ $($i.e., non-Lipschitz continuous$)$.
Let $\sets{(x^k, \breve{y}^k)}$ be generated by Algorithm~\ref{alg:A1} using the following update rules:  
\myeq{eq:update_rule}{
\left\{\begin{array}{l}
	\tau_k := \dfrac{1}{k + 1},\quad \rho_k := \dfrac{\rho_0}{\tau_k},\quad \eta_k := (1 - \gamma)\rho_k, \vspace{1ex}\\
	L_k :=  L_f + L_g M_H + \dfrac{M_g^2\rho_k}{\gamma}, \quad \text{and} \quad \beta_{k+1} := \dfrac{(1-\tau_k)\tau_{k+1}}{\tau_k},
\end{array}\right.
}
where $\rho_0 > 0$ is a given and $\gamma \in (0, 1)$ is fixed.
Then, the following bounds hold:
\myeq{eq:BigO_nonegrodic_rate}{
\left\{\arraycolsep=0.2em
\begin{array}{rlcl}
	& \Pc(x^k) - \Pc^{\star} & \leq & \dfrac{1}{2k} \Big[ L_0 \norms{x^0 - x^{\star} }^2 + \dfrac{1}{(1-\gamma)\rho_0} \big( \norms{y^0} + M_H \big)^2 \Big],\vspace{0.5ex}\\
	& \Dc(\breve{y}^k) - \Dc^{\star} & \leq & \dfrac{1}{2k} \Big[ L_0 \big(\norms{x^0} + M_{F^{*}} \big)^2 + \dfrac{1}{(1-\gamma)\rho_0} \norms{y^0 - y^{\star}}^2 \Big], \vspace{0.5ex}\\
	& \Gc_{\Xc \times \Yc} (x^k, \breve{y}^k) & \leq & \dfrac{\Rc^2_{\Xc\times\Yc}}{2k}, 
\end{array}\right.
}
where $\Gc_{\Xc \times \Yc}$ is defined by \eqref{eq:gap_func}, $\Rc^2_{\Xc\times\Yc} := \sup\set{ L_0\norms{x - x^0}^2 + \frac{1}{\eta_0}\norms{y - y^0}^2 : x \in \Xc, y \in\Yc}$, and $M_H,\ M_{F^*} \in [0, \ \infty]$ are the Lipschitz constants of $H$ and $F^*$, respectively.
\end{theorem}

\begin{remark}\label{rmk:Bg}
Again, the right-hand side of the primal (respectively, the dual) convergence rate bound in Theorem~\ref{thm:BigO_nonegrodic_rate} is finite if $M_H$ (respectively, $M_{F^{*}})$ is finite.
\end{remark}

\subsubsection{The strongly convex case}
We again consider the case $F$ is strongly convex, i.e., $\mu_F := \mu_f + \mu_h > 0$.
The following theorem establishes the $\BigO{1/k^2}$ semi-ergodic optimal convergence rate of Algorithm~\ref{alg:A1}, whose proof can be found in Appendix \ref{subsec:alg2_non_ergodic_proof}.
			
\begin{theorem}\label{thm:BigO_nonegrodic_rate_str}
Suppose that Assumptions \ref{as:A1} and \ref{as:A2} hold for \eqref{eq:primal_cvx}, and $F$ is $\mu_F$-strongly convex with $\mu_F > 0$.
Suppose, in addition, that $H$ is $M_H$-Lipschitz continuous with a Lipschitz constant such that $L_gM_H <+ \infty$.
In particular, if $g$ is affine, then $L_g = 0$ and we allow $M_H = +\infty$. 
Let $\sets{(x^k, \breve{y}^k)}$ be generated by Algorithm \ref{alg:A1} using the following updates: 
\myeq{eq:update_rule_str}{
\hspace{-1ex}
\arraycolsep=0.1em
\left\{\begin{array}{l}
\tau_{k + 1} := \dfrac{\tau_k}2 \left(\sqrt{\tau_k^2 + 4} - \tau_k\right),  \quad \rho_k := \dfrac{\rho_0}{\tau_k^2}, \quad \eta_k := (1 - \gamma)\rho_k, \vspace{1ex}\\
L_k :=  L_f + L_gM_H + \frac{M_g^2\rho_k}{\gamma}, \ \ \text{and} \ \ 
\beta_{k+1} := \dfrac{(1-\tau_k)\tau_k(L_k + \mu_h)}{\tau_k^2(L_k+\mu_h) + (L_{k+1} + \mu_h)\tau_{k+1}}.
\end{array}\right.
\hspace{-2ex}
}
where $\tau_0 := 1$,  \ $0 < \rho_0 \leq \frac{\mu_F}{L_gM_H + M_g^2}$, and $\gamma \in (0, 1)$ is fixed.
Then, we have 
\myeq{eq:BigO_rate_str}{
		\left\{\arraycolsep=0.2em
		\begin{array}{rlcl}
			& \Pc(x^k) - \Pc^{\star} & \leq & \dfrac{2}{{(k + 1)}^2} \Big[ L_0\norms{x^0 - x^{\star}}^2  + \dfrac{1}{(1-\gamma)\rho_0} \big( \norms{y^0} + M_H \big)^2 \Big],\vspace{0.5ex}\\
			& \Dc(\breve{y}^k) -  \Dc^{\star} & \leq & \dfrac{2}{{(k + 1)}^2} \Big[ L_0 \big( \norms{x^0} + M_{F^{*}} \big)^2 + \dfrac{1}{(1-\gamma)\rho_0}\norms{y^0 - y^{\star}}^2 \Big], \vspace{0.5ex}\\
			& \Gc_{\Xc \times \Yc} (x^k, \breve{y}^k) & \leq & \dfrac{2\Rc^2_{\Xc\times\Yc}}{{(k + 1)}^2}, 
		\end{array}\right.
}
where  $\Gc_{\Xc \times \Yc}$ is defined by \eqref{eq:gap_func}, $\Rc^2_{\Xc\times\Yc}$ is defined in Theorem~\ref{thm:BigO_nonegrodic_rate}, and $M_H,\ M_{F^{*}} \in [0, +\infty]$ are the Lipschitz  constants of $H$ and $F^*$, respectively. 
\end{theorem}

\begin{remark}\label{re:optimal_rates}
The $\BigO{1/k}$ and $\BigO{1/k^2}$ convergence rates stated in this paper are optimal (up to a constant factor) under the corresponding assumptions in the above four theorems since these rates are optimal for the special linear case as shown in \cite{tran2019non}.
\end{remark}

Note that the update of $L_k$ in \eqref{eq:update_rule}  of Theorem~\ref{thm:BigO_nonegrodic_rate} and \eqref{eq:update_rule_str} of Theorem~\ref{thm:BigO_nonegrodic_rate_str} does not require an upper bound constant $D$ as in Theorems~\ref{thm:BigO_ergodic_rate} and \ref{thm:alg2_ergodic}.
Moreover, the momentum step-size $\beta_{k+1}$ in \eqref{eq:update_rule_str} is new as we have not seen it in the literature.

\section{Application to Convex Cone Constrained Optimization}\label{sec:extension_general_phi}
In this section, we specify Algorithm~\ref{alg:A1} and their convergence results to handle the special case \eqref{eq:cone_cvx_opt} of \eqref{eq:saddle_point}.
This problem is a general convex cone constrained program as in \cite{hamedani2018primal}, and is more general than the setting studied in other existing works, e.g., \cite{xu2017first}.
By Assumption \ref{as:A2}(b), since $\iprods{y,\ g(x)}$ is convex in $x$ for any $y \in \Kc^{*}$, $g$ is $\Kc$-convex, i.e., for all $x, \hat{x}\in\dom{g}$ and $\lambda \in [0,1]$, it holds that $(1-\lambda)g(x) + \lambda g(\hat{x}) - g\left((1 - \lambda)x + \lambda\hat{x}\right)\in\Kc$. 
To guarantee strong duality, we require the Slater condition on \eqref{eq:cone_cvx_opt}: $\set{x \in \ri{\dom{F}} : g(x) \in -\mathrm{int}(\Kc)}\neq\emptyset$.

To solve \eqref{eq:cone_cvx_opt}, we apply Algorithm~\ref{alg:A1} and replace the update of $y^{k+1}$ at Step \ref{step:A1:s1} by
\myeq{eq:projec_yplus}{
	 y^{k + 1} := \proj_{\Kc^{*}} \left(\tilde{y}^k + \rho_k g(\hat{x}^k)\right), 
}
where $\proj_{\Kc^{*}}$ is the projection onto the dual cone $\Kc^{*}$.
We will characterize the convergence of Algorithm~\ref{alg:A1} via the following combined primal-dual measurement:
\myeq{eq:Ec_gap}{
	\Ec(x) :=  \max\big\{ \vert F(x) - F^{\star}\vert,\ \kdist{-\Kc}{g(x)} \big\},
}
where $\kdist{-\Kc}{g(x)} := \inf_{s \in -\Kc} \norm{g(x) - s}$ is the Euclidean distance from $g(x)$ to $-\Kc$. 

The following theorem proves the convergence of the proposed variant of Algorithm~\ref{alg:A1} for solving \eqref{eq:cone_cvx_opt}, whose proof can be found in Appendix \ref{app:conic}.

\begin{theorem}\label{thm:conic}
Suppose that Assumptions~\ref{as:A1} and \ref{as:A2}, and the Slater condition hold for \eqref{eq:cone_cvx_opt}.
Let $\sets{x^k}$ be generated by the variant of Algorithm~\ref{alg:A1} using \eqref{eq:projec_yplus} for solving \eqref{eq:cone_cvx_opt}.
Let $\Ec(x)$ be defined by \eqref{eq:Ec_gap} and $\Delta_0 := L_0\norms{x^0 - x^{\star}}^2 + \frac{1}{\eta_0} \left(\norms{y^0} + \norms{y^{\star}} + 1\right)^2$.
Then, we have:
\begin{enumerate}
\item[$\mathrm{(a)}$] Under the conditions of Theorem \ref{thm:BigO_ergodic_rate}, we have  $\Ec(\bar{x}^{k}) \leq \tfrac{\Delta_0}{2k}$ in ergodic.
\item[$\mathrm{(b)}$] Under the conditions of Theorem \ref{thm:alg2_ergodic}, we have $\Ec(\bar{x}^{k}) \leq \tfrac{\Delta_0}{2[\rho_0k + P_0k(k-1)]}$ in ergodic.
\end{enumerate}
Alternatively, if $g$ is either affine or  bounded by $B_g$ on $\dom{F}$ $($i.e., $\norms{g(x)} \leq B_g$ for all $x\in\dom{F}$, or in particular, $\dom{F}$ is bounded$)$, and Algorithm \ref{alg:A1} uses $L_k := L_f + \frac{1}{\gamma}L_g\left(L_g[\norms{y^0}/\rho_0 + (2-\gamma)B_g] + M_g^2 \right)$ at Step~\ref{step:A1:main_step1}, then the following guarantees hold:
\begin{enumerate}
\item[$\mathrm{(c)}$] Under the conditions of Theorem \ref{thm:BigO_nonegrodic_rate}, we have $\Ec(x^{k}) \leq \tfrac{\Delta_0}{2k}$ in the last iterate $x^k$.
\item[$\mathrm{(d)}$] Under the conditions of Theorem \ref{thm:BigO_nonegrodic_rate_str}, we have $\Ec(x^{k}) \leq \tfrac{2\Delta_0}{{(k + 1)}^2}$ in the last iterate $x^k$.
\end{enumerate}
\end{theorem}
In the last two cases (c) and (d) of Theorem~\ref{thm:conic}, we do not have the Lipschitz continuity of $H$ since $H(\cdot) = \delta_{-\Kc}(\cdot)$.
Hence, we need to derive a new upper bound for $\Lb_g^k$ of Lemma~\ref{lm:descent_property1} to update $L_k$.
The convergence bounds of Theorem~\ref{thm:conic} already combine both the primal objective residual $\vert F(x) - F^{\star}\vert$ and the primal feasibility violation $\kdist{-\Kc}{g(x)}$.
Moreover, their convergence rates are optimal.
The statements (a) and (b) cover \cite{xu2017first} as  special cases.

\section{Numerical Experiments}\label{sec:num_exp}
In this section, we test and compare different variants of Algorithm~\ref{alg:A1} on two numerical examples. 
The first one is a binary classification with multiple distributions in Subsection \ref{subsec:robust_learn}.
The second example is a convex-linear minimax game in Subsection \ref{subsec:game}.
Our experiments are implemented in Matlab R2018b, running on a Laptop with 2.8 GHz Quad-Core Intel Core i7 and 16Gb RAM using Microsoft Windows.

%

\subsection{Binary classification with multiple distributions}\label{subsec:robust_learn}
Let $\Pbb_1, \cdots, \Pbb_n$ denote $n$ distributions of $n$ datasets, respectively.
We consider the following optimization problem, which we call a binary classification with multiple distributions:
\myeq{eq:robust_learn}{
\min_{x \in \R^p}\Big\{ \max_{1\leq i \leq n} \expect{a \sim \Pbb_i}{G(x, a)} + \rho R(x) \Big\}, 
}
where $R$ is a given convex regularizer, $\rho > 0$, $G(x, a)$ is a convex loss function depending on a parameter $x$ and an input data $a$. 
Let $g_i(x) := \expect{a \sim \Pbb_i}{G(x, a)}$ denote the expected loss over the distribution $\Pbb_i$, and $H(u) := \max_{1\leq i \leq n}u_i$.
Then, \eqref{eq:robust_learn} can be cast into \eqref{eq:primal_cvx}.

In this experiment, we choose $R(x) := \frac{1}{2}\norms{x}^2$, $G(x, a) := \log(1 + \exp(1 + a^{\top}x))$, which is a common logistic loss function widely used in binary classification. 
Here, $a$ is the product of the feature vector $w \in \R^p$ and the label $z$ for each example.
After some experiment, we find that $\rho := 0.01$ is a reasonable regularization parameter value.

Since $f$ is strongly convex, we implement four variants of Algorithm \ref{alg:A1}: \texttt{Alg.1 (v1)}, \texttt{(v2)}, \texttt{(v3)}, and \texttt{(v4)}, and compare them with the Accelerated Primal-Dual (\texttt{APD}) algorithm (the strongly convex variant)  in \cite{hamedani2018primal}, and the \texttt{Mirror Descent} method in \cite{Nemirovskii2004}.
Note that  \texttt{Mirror Descent} is double-loop where the inner loop approximately computes the prox-mapping. 
Since we have tuned the parameters,  \texttt{Mirror Descent} implicitly exploits the strong convexity of $f$ in \eqref{eq:robust_learn}.
We use different datasets from \texttt{LIBSVM} \cite{CC01a} to create five tests. 
For the first test, we treat $9$ datasets: \texttt{a1a}, $\cdots$, \texttt{a9a} as observations generated from $9$ different distributions, respectively. 
Similarly, in the second test, we use $8$ datasets: \texttt{w1a}, $\cdots$, \texttt{w8a} as observations generated from $8$ different distributions, respectively. 
For the last three tests, we use three datasets: \texttt{covtype}, \texttt{rcv1}, and \texttt{news20} and split their total observations into $10$ blocks and assume that each block is generated from a different distribution.

\begin{table}[hpt!]
\newcommand{\cell}[1]{{\!\!}#1{\!\!}}
\newcommand{\cellb}[1]{{{\!\!}\color{blue}#1{\!\!}}}
\caption{The numerical results of six algorithms on the five tests of problem \eqref{eq:robust_learn} after $1000$ iterations.}\label{table:robust_learn}

\resizebox{\textwidth}{!}{
\begin{tabular}{| l | rr | rr | rr | rr | rr | rr | rr |}
\hline
\multicolumn{1}{|c}{{\!\!}Test Prob.{\!\!}} & \multicolumn{2}{|c}{Size }  &  \multicolumn{2}{|c}{\texttt{Alg.1 (v1)}}  &     \multicolumn{2}{|c}{\texttt{Alg.1 (v2)}}  &  \multicolumn{2}{|c|}{\texttt{Alg.1 (v3)}} &  \multicolumn{2}{|c|}{\texttt{Alg.1 (v4)}} &  \multicolumn{2}{|c|}{\texttt{APD}} &  \multicolumn{2}{|c|}{\!\!\texttt{Mirror Descent}{\!\!}}        \\ \hline
  &\cell{$n$}   & \cell{$p$}       & \cell{gap~} & \cell{time[s]} & \cell{gap~} & \cell{time[s]} & \cell{gap~} & \cell{time[s]} & \cell{gap~} & \cell{time[s]} & \cell{gap~} & \cell{time[s]} & \cell{gap~} & \cell{time[s]}    \\ \hline
\texttt{a1a-a9a}  & \cell{9}   & \cell{123} & \cell{5.3e-3}   & \cell{43.00} & \cell{3.4e-4}   & \cell{34.34} & \cell{3.9e-3}   & \cell{42.56} & \cellb{1.8e-4}   & \cellb{29.45} & \cell{4.2e-4}   & \cell{30.48} & \cell{2.9e-4}   & \cell{57.86}    \\ \hline
\texttt{w1a-w8a}  & \cell{8}   & \cell{300} & \cell{1.5e-3}   & \cell{61.56} & \cell{3.7e-4}   & \cell{56.09} & \cell{1.1e-3}   & \cell{62.98} & \cellb{1.0e-4}   & \cell{44.68} & \cell{1.8e-4}   & \cellb{43.67} & \cell{5.1e-4}   & \cell{107.36}    \\ \hline
\texttt{covtype}  & \cell{10}   & \cell{54} & \cell{1.3e-2}   & \cell{110.04} & \cell{3.3e-5}   & \cell{90.20} & \cell{2.1e-3}   & \cell{106.44} & \cell{6.4e-5}   & \cellb{63.93} & \cellb{3.0e-5}   & \cell{86.83} & \cell{1.4e-4}   & \cell{141.11}    \\ \hline
\texttt{rcv1}  & \cell{10}   & \cell{47,236} & \cell{2.1e-3}   & \cell{25.40} & \cell{1.2e-4}   & \cell{23.87} & \cell{6.7e-4}   & \cell{24.12} & \cellb{1.3e-5}   & \cellb{23.24} & \cell{1.5e-5}   & \cell{23.39} & \cell{3.2e-4}   & \cell{48.17}    \\ \hline
\texttt{news20}  & \cell{10}   & \cell{1,355,191} & \cell{2.4e-3}   & \cell{242.05} & \cell{2.1e-5}   & \cell{261.58} & \cell{1.4e-4}   & \cell{241.99} & \cellb{6.2e-6}   & \cellb{232.06} & \cell{3.2e-4}   & \cell{232.59} & \cell{1.3e-4}   & \cell{485.05}    \\ \hline
\end{tabular}}
\vspace{-1ex}
\end{table}

To obtain a fair comparison, we tune the hyper-parameters of the underlying algorithms.
For \texttt{APD} and \texttt{Mirror Descent}, we tune their primal-dual step-sizes in the range of $[0.001, 0.01, 0.1, 1, 10]$. 
For our algorithms, we set $\gamma := \tfrac{1}{2}$ and tune $\rho_0$ in the range of $[0.001, 0.01, 0.1, 1, 10]$. 
We use both the primal-dual (or duality) gap $\Pc(x^k) - \Dc(y^k)$ and the CPU time (in seconds)  to measure algorithm's performance.
Our numerical results on the $5$ tests are summarized in Table \ref{table:robust_learn} after $1000$ iterations. 
Note that, from theoretical convergence guarantees, the gap is computed based on averaging sequences, i.e., $\Pc(\bar{x}^k) - \Dc(\bar{y}^k)$ for \texttt{Alg.1 (v1)}, \texttt{(v2)}, \texttt{APD}, and \texttt{Mirror Descent}.
For  \texttt{Alg.1 (v3)} and \texttt{(v4)}, it is computed based on the primal last iterate and the dual averaging sequence, i.e., $\Pc(x^k) - \Dc(\bar{y}^k)$.

From Table \ref{table:robust_learn}, we observe that, overall, \texttt{Alg.1 (v4)} outperforms other algorithms in terms of the primal-dual gap and CPU time in most cases. 
Since \texttt{Alg.1 (v1)} and \texttt{Alg.1 (v3)} do not utilize the strong convexity of $f$ in \eqref{eq:robust_learn}, their performance is  worse than that of other four methods, where the strong convexity is exploited. 
Finally, to illustrate the progress of $6$ algorithms, Figure \ref{fig:robust_learn} shows their primal-dual gap vs. iterations on two tests: \texttt{a1a-a9a} (left plot) and \texttt{news20} (right plot).
\begin{figure*}[hpt!]
\vspace{-0.5ex}
\begin{center}
\includegraphics[width=0.45\textwidth]{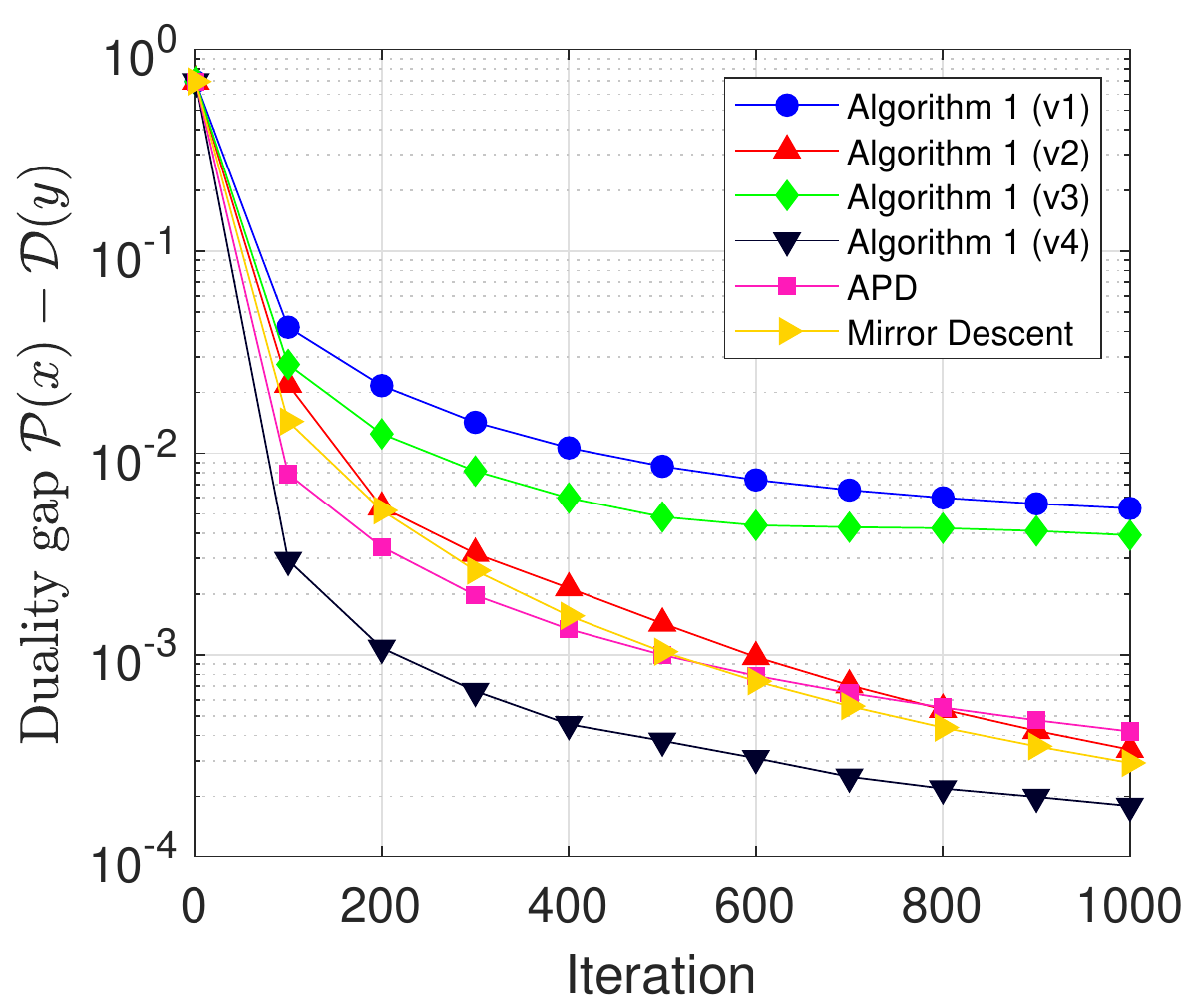}
\includegraphics[width=0.45\textwidth]{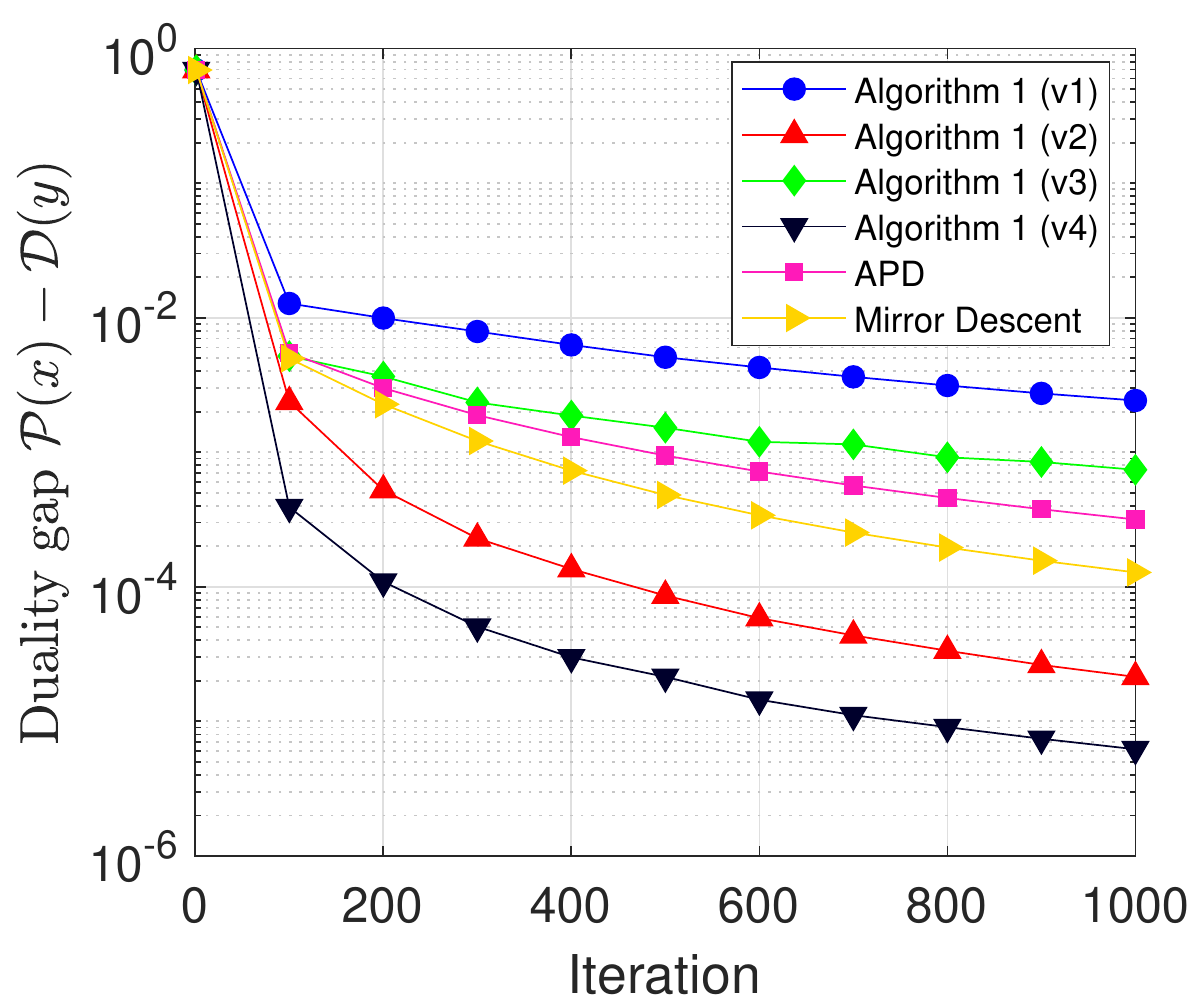}
\vspace{-1.5ex}
\caption{The convergence behavior of six algorithms for solving \eqref{eq:robust_learn} on \texttt{a1a-a9a} $($left$)$ and \texttt{news20} $($right$)$.}
\label{fig:robust_learn}
\vspace{-2ex}
\end{center}  
\end{figure*}

\subsection{Convex-concave minimax game}\label{subsec:game}
We consider a convex-concave minimax game between two players, 
where Player 1 chooses her strategy $x \in \Delta_p := \sets{x \in \R_{+}^p \, : \, \sum_{j = 1}^p x_j = 1}$ to minimize a cost function $F(x)$, and simultaneously, 
Player 2 chooses her strategy $y \in \Delta_n := \sets{y \in \R_{+}^n \, :\,  \sum_{i = 1}^n y_i = 1}$ to minimize a cost function $H^{*}$.
In addition, Player 1 has to pay $\Phi(x, y)$ loss to Player 2.
By concrete choices of $\Phi$ as in \cite[Section 4.3]{chen2017accelerated}, we can model this problem into the following minimax problem with convex-linear coupling term:
\myeq{eq:exp_game}{
	\min_{x \in \Delta_p} \max_{y \in \Delta_n} \bigg\{ \widetilde{\Lc}(x, y) := \frac{1}{N}\sum_{j = 1}^N \text{log} (1 + \exp( a_j^{\top} x)) + \sum_{i = 1}^n \frac{b_i y_i}{1 + x_i} \bigg\}.
}
Clearly, \eqref{eq:exp_game} can be cast into our model \eqref{eq:saddle_point} with $f(x) := \frac{1}{N}\sum_{j = 1}^N \text{log} (1 + e^{a_j^{\top} x})$ and $g_i(x) := \frac{b_i}{1+x_i}$.
We can compute $L_f := \frac{1}{4}\norms{A}^2$, and $L_{g_i} := 2\vert b_i\vert$, $M_{g_i} := \vert b_i\vert$ for $i \in [n]$.

Since $f$ in \eqref{eq:exp_game} is not strongly convex, we solve \eqref{eq:exp_game} using two variants of Algorithm \ref{alg:A1}: \texttt{Alg.1 (v1)} and \texttt{(v3)} both with $\Oc(1/k)$ convergence rates on the primal-dual gap.
We again compare \texttt{Alg.1 (v1)}, \texttt{(v3)} with the Accelerated Primal-Dual (\texttt{APD}) algorithm \cite{hamedani2018primal}, and the \texttt{Mirror Descent} method \cite{Nemirovskii2004} (without strong convexity).
The hyper-parameters of these algorithms are tuned in the same way as in Section \ref{subsec:robust_learn}.

For input data, we take the \texttt{real-sim} dataset from \texttt{LIBSVM}  \cite{CC01a} to form vector $a_i$ and generate $b_i$ by using a standard uniform distribution. 
To fully test the performance of four algorithms, we generate $30$ problem instances by randomly splitting the \texttt{real-sim} dataset  into $30$ equal blocks ($N=2411$ samples per block), respectively.
The performance of the four algorithms on $30$ problem instances is depicted in Figure~\ref{fig:game_real-sim}. 
Here, for \texttt{Alg.1 (v1)}, \texttt{APD}, and \texttt{Mirror Descent}, the duality gap is computed based on both the primal and dual averaging sequences, i.e., $\Pc(\bar{x}^k) - \Dc(\bar{y}^k)$, 
while for \texttt{Alg.1 (v3)}, this gap is computed through the primal last iterate and the dual averaging sequence, i.e., $\Pc(x^k) - \Dc(\bar{y}^k)$.
Next, we compute the statistic mean over all $30$ instances and highlight it in a thick curve of Figure~\ref{fig:game_real-sim}, while the deviation range of the duality gap and CPU time is plotted in a shaded area.

From Figure~\ref{fig:game_real-sim}, we observe that \texttt{Alg.1~(v3)} converges faster than \texttt{Alg.1~(v1)}, \texttt{Mirror descent}, and \texttt{APD}. However, it exhibits the most oscillation behavior as shown through both the mean curve and the shaded deviation area. 
In fact, this is a normal behavior since it uses the last-iterate sequence which does not have a monotone decrease on the duality gap, and thus is less smooth than other curves, which use an averaging sequence.

\begin{figure}[hpt!]
\vspace{-0.5ex}
\begin{center}
	\includegraphics[width=0.45\textwidth]{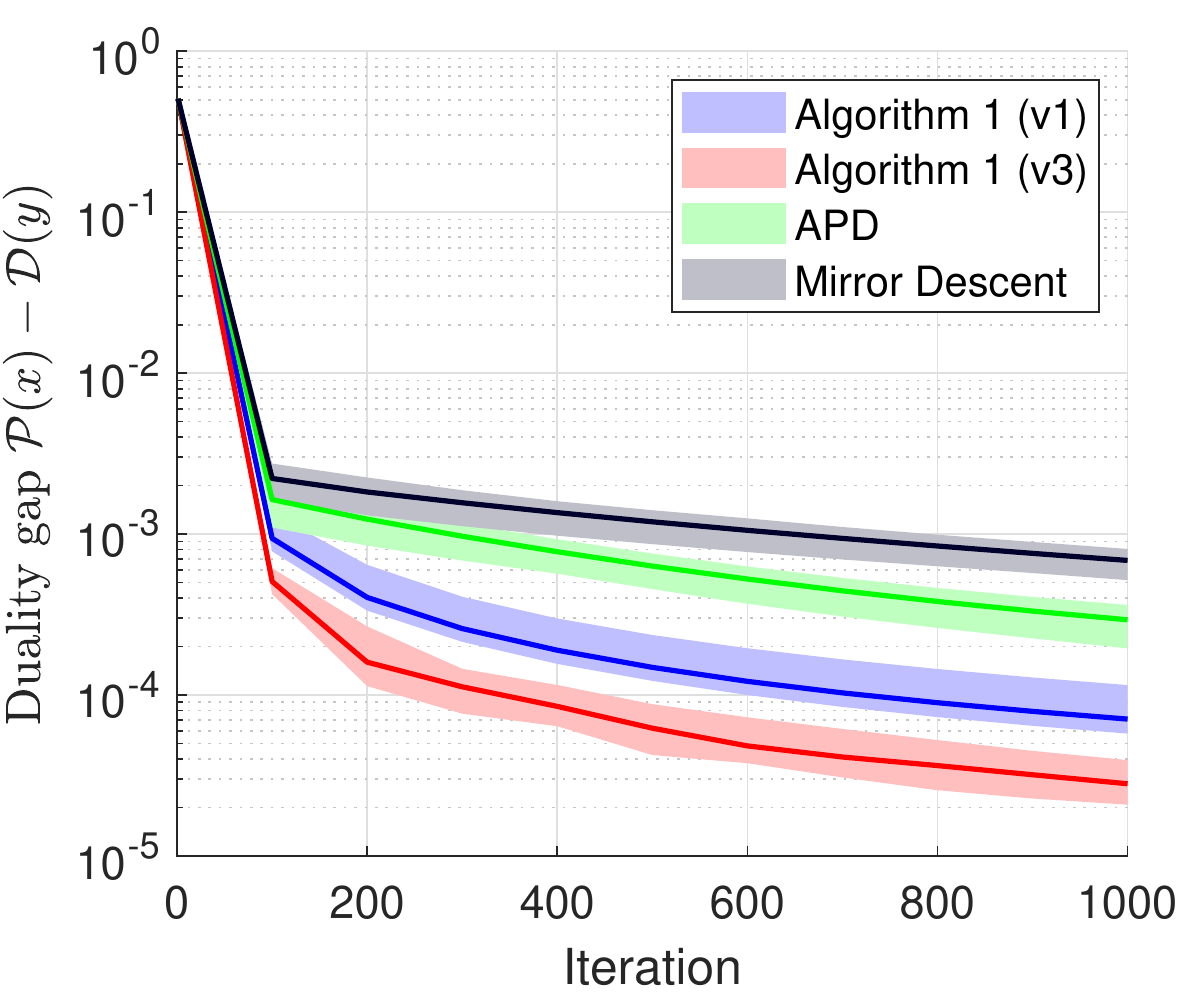}
	\includegraphics[width=0.45\textwidth]{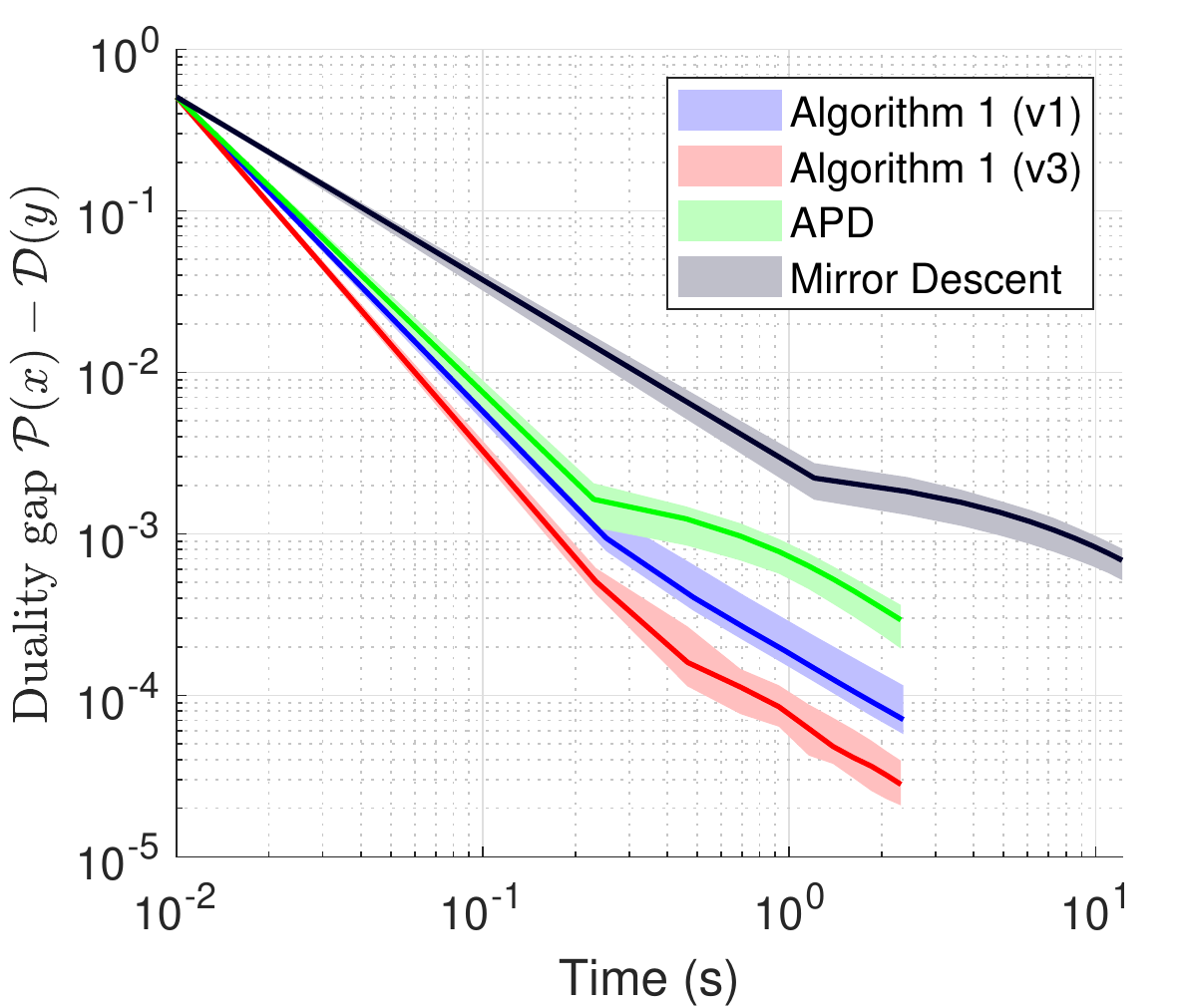}
	\vspace{-1.5ex}
	\caption{The average performance of the $4$ algorithms on $30$ problem instances of \eqref{eq:exp_game} using the \texttt{real-sim} dataset. 
	Left: Duality gap against iteration counter. Right: Duality gap against CPU time.}
	\label{fig:game_real-sim}
\vspace{-2ex}	
\end{center}
\end{figure}

%


\appendix

\section{ One-Iteration Analysis  of Algorithm~\ref{alg:A1}}
This appendix provides the full proofs of Lemma~\ref{lm:descent_property1}.

\subsection{General bound}\label{subsec:A_alg1_keylm}
We first prove the following general bound for our method.

\begin{lemma}\label{le:descent_property1}
Let $s^{k+1}$, $x^{k+1}$, and $\tilde{y}^{k+1}$ be updated by  \eqref{eq:scheme_s}, \eqref{eq:scheme_x}, and \eqref{eq:scheme_ytilde}, respectively.
Define $\Lb_g^k := \Lb_g\big( \tilde{y}^k + \rho_k [g(\hat{x}^k) - s^{k + 1}] \big)$.
Then, for any $(x, s)\in\dom{F}\times\dom{H}$, we have {\!\!\!}
\myeq{eq:keylemma_Phi2}{
\arraycolsep=0.1em
\begin{array}{rl}
	\Lc_{\rho_k}\!& (x^{k + 1}, s^{k + 1}, \tilde{y}^k) \ \leq \ \Lc_{\rho_k} (x, s, \tilde{y}^k) + L_k \iprods{x^{k + 1} - \hat{x}^k,\ x - x^{k + 1}}\vspace{1ex}\\
	& - {~} \frac{\rho_k}{2} \norms{[g(x) - s] - [g(\hat{x}^k) - s^{k + 1}]}^2 + \frac{1}{2}\big( \Lb_g^k  + L_f + \rho_k M_g^2\big) \norms{x^{k + 1} - \hat{x}^k}^2 \vspace{1ex}\\
	& - {~} \frac{\mu_h}{2}\norms{x^{k+1} - x}^2 - \frac{\mu_f}{2}\norms{\hat{x}^k - x}^2 - \frac{\mu_{H^{*}}}{2}\norms{\nabla{H}(s^{k+1}) - \nabla{H}(s)}^2.  
\end{array}
}
\end{lemma}

\begin{proof}
First, the optimality condition of the $x^{k+1}$-subproblem~\eqref{eq:scheme_x}, can be written as
\myeq{eq:f_opt}{
	0 =  \nabla h(x^{k + 1})  + \nabla{f}(\hat{x}^k) +   \nabla_x \phi_{\rho_k} (\hat{x}^k, s^{k + 1}, \tilde{y}^k)  + L_k( x^{k + 1} - \hat{x}^k),
}
for some $\nabla{h}(x^{k + 1}) \in \partial{h}(x^{k + 1})$.
Next, by the convexity of $h$ and $f$, and the $L_f$-smoothness of $f$, for any $x \in\dom{F}$, we have
\vspace{-0.5ex}
\myeqn{
\left\{\arraycolsep=0.2em
\begin{array}{lcl}
	h(x^{k + 1}) & \leq & h(x) + \iprods{\nabla{h}(x^{k + 1}),\ x^{k + 1} - x} - \frac{\mu_h}{2}\norms{x^{k+1} - x}^2,\vspace{1ex}\\
	f(x^{k+1}) & \leq & f(x) + \iprods{\nabla{f}(\hat{x}^k),\ x^{k+1} - x} + \frac{L_f}{2}\norms{x^{k+1} - \hat{x}^k}^2 - \frac{\mu_f}{2}\norms{\hat{x}^k - x}^2.
\end{array}\right.
\vspace{-0.5ex}
}
Combining these two inequalities and then using \eqref{eq:f_opt} and $F := f + h$, we can derive
\myeq{eq:p2_keylemma_proof_optimality}{
\begin{array}{lcl}
	F(x^{k + 1})   &\stackrel{\eqref{eq:f_opt}}\leq&  F(x) - \iprods{\nabla_x \phi_{\rho_k} (\hat{x}^k, s^{k + 1}, \tilde{y}^k),\ x^{k + 1} - x} + \frac{L_f}{2}\norms{x^{k+1} - \hat{x}^k}^2  \vspace{1ex}\\
	&& + {~}  L_k \iprods{x^{k + 1} - \hat{x}^k,\ x -  x^{k + 1}} - \frac{\mu_h}{2}\norms{x^{k+1} - x}^2 - \frac{\mu_f}{2}\norms{\hat{x}^k - x}^2.
\end{array}
}
Similarly, using the optimality condition of \eqref{eq:scheme_s} and the convexity and $\frac{1}{\mu_{H^{*}}}$-smoothness of $H$ (we allow $\mu_{H^{*}}$ to be zero, corresponding to the nonsmoothness of $H$), we have
\myeq{eq:keylm_s_descent}{
\arraycolsep=0.2em
\begin{array}{lcl}
	H(s^{k + 1}) & \leq & H(s) + \iprods{\nabla_s \phi_{\rho_k} (\hat{x}^k, s^{k + 1}, \tilde{y}^k),\ s - s^{k + 1}}   - \frac{\mu_{H^{*}}}{2}\norms{\nabla{H}(s^{k+1}) - \nabla{H}(s)}^2.  
\end{array}	
}
Using \eqref{eq:pd_Lipschitz_local_onlynl} in Lemma~\ref{lm:Lipschitz_continuous} and $\Lb_g^k := \Lb_g \big( \tilde{y}^k + \rho_k [g(\hat{x}^k) - s^{k + 1}] \big)$, for any $(x, s)$, we get
\myeqn{
\left\{\arraycolsep=0.2em
\begin{array}{lcl}
	\phi_{\rho_k} (x^{k + 1}, s^{k + 1}, \tilde{y}^k) & \leq & \phi_{\rho_k} (\hat{x}^k, s^{k + 1}, \tilde{y}^k) + \iprods{\nabla_x \phi_{\rho_k} (\hat{x}^k, s^{k + 1}, \tilde{y}^k),\ x^{k + 1} - \hat{x}^k}\vspace{1ex}\\
	& & + {~}  \frac{\rho_k}{2} \norms{g(x^{k + 1}) - g(\hat{x}^k)}^2 + \frac{\Lb_g^k}2 \norms{x^{k + 1} - \hat{x}^k}^2,\vspace{1.5ex}\\
	\phi_{\rho_k} (x, s, \tilde{y}^k) & \geq & \phi_{\rho_k} (\hat{x}^k, s^{k + 1}, \tilde{y}^k) + \iprods{\nabla_x \phi_{\rho_k} (\hat{x}^k, s^{k + 1}, \tilde{y}^k),\ x - \hat{x}^k}\vspace{1ex}\\
	& & + {~}  \iprods{\nabla_s \phi_{\rho_k} (\hat{x}^k, s^{k + 1}, \tilde{y}^k),\ s - s^{k + 1}}  +  \frac{\rho_k}{2} \norms{[g(x) - s] - [g(\hat{x}^k) - s^{k + 1}]}^2.
\end{array}\right.
}
By \eqref{eq:M_g}, the above two inequalities imply
\myeq{eq:p2_keylemma_proof_psi}{
\arraycolsep=0.3em
\hspace{-1ex}
\begin{array}{lcl}
	\phi_{\rho_k}(x^{k + 1}, s^{k + 1}, \tilde{y}^k) & \leq & \phi_{\rho_k} (x, s, \tilde{y}^k) + \iprods{\nabla_x \phi_{\rho_k} (\hat{x}^k, s^{k + 1}, \tilde{y}^k),\ x^{k + 1} - x}\vspace{1ex}\\
	&& + {~}  \iprods{\nabla_s \phi_{\rho_k} (\hat{x}^k, s^{k + 1}, \tilde{y}^k),\ s^{k + 1} - s}  + \frac{\Lb_g^k + \rho_k M_g^2}{2} \norms{x^{k + 1} - \hat{x}^k}^2 \vspace{1ex}\\
	&& - {~}  \frac{\rho_k}{2} \norms{[g(x) - s] - [g(\hat{x}^k) - s^{k + 1}]}^2.
\end{array}
\hspace{-1ex}
}
Now, combining \eqref{eq:p2_keylemma_proof_optimality}, \eqref{eq:keylm_s_descent}, and \eqref{eq:p2_keylemma_proof_psi}, we can derive
\myeqn{
\arraycolsep=0.2em
\begin{array}{lcl}
	\Lc_{\rho_k}(x^{k + 1}, s^{k + 1}, \tilde{y}^k) & = & F(x^{k + 1}) + H(s^{k + 1}) + \phi_{\rho_k} (x^{k + 1}, s^{k + 1}, \tilde{y}^k)\vspace{1ex}\\
	& \leq & F(x) + H(s) + \phi_{\rho_k} (x, s, \tilde{y}^k) + L_k \iprods{x^{k + 1} - \hat{x}^k,\ x - x^{k + 1}}\vspace{1ex}\\
	&& - {~}  \frac{\rho_k}{2} \norms{[g(x) - s] - [g(\hat{x}^k) - s^{k + 1}]}^2 + \frac{\Lb_g^k + L_f + \rho_k M_g^2}{2} \norms{x^{k + 1} - \hat{x}^k}^2 \vspace{1ex}\\
	&& - {~} \frac{\mu_h}{2}\norms{x^{k+1} - x}^2 - \frac{\mu_f}{2}\norms{\hat{x}^k - x}^2 - \frac{\mu_{H^{*}}}{2}\norms{\nabla{H}(s^{k+1}) - \nabla{H}(s)}^2,
\end{array}
}
which proves \eqref{eq:keylemma_Phi2}.
\end{proof}

\subsection{The proof of Lemma \ref{lm:descent_property1}}
First, let us recall the following simple expressions: {\!\!\!}
\myeqn{
\arraycolsep=0.1em
\begin{array}{lcl}
	(1-\tau_k)\norms{x^{k+1} {\!} - x^k}^2 + \tau_k\norms{x^{k+1} {\!} - x}^2 &= & \norms{x^{k+1} {\!} - (1-\tau_k)x^k - \tau_kx}^2 + \tau_k(1-\tau_k)\norms{x^k - x}^2, \vspace{1.25ex}\\
	(1-\tau_k)\norms{\hat{x}^{k} - x^k}^2 + \tau_k\norms{\hat{x}^{k} - x}^2 & = & \norms{\hat{x}^{k} - (1-\tau_k)x^k - \tau_kx}^2 + \tau_k(1-\tau_k)\norms{x^k - x}^2. 
\end{array}
}
Plugging $(x, s) := (x^k, s^k)$ in \eqref{eq:keylemma_Phi2} of Lemma \ref{le:descent_property1}, we obtain
\myeqn{
\arraycolsep=0.2em
\begin{array}{lcl}
	\Lc_{\rho_k} (x^{k + 1}, s^{k + 1}, \tilde{y}^k) & \leq & \Lc_{\rho_k} (x^k, s^k, \tilde{y}^k) + L_k \iprods{x^{k + 1} - \hat{x}^k,\ x^k - x^{k + 1}}\vspace{1ex}\\
	& & - {~}  \frac{\rho_k}{2} \norms{[g(x^k) - s^k] - [g(\hat{x}^k) - s^{k + 1}]}^2 + \frac{\Lb_g^k + L_f + \rho_k M_g^2}{2} \norms{x^{k + 1} - \hat{x}^k}^2 \vspace{1ex}\\
	&& - {~} \frac{\mu_f}{2}\norms{\hat{x}^k - x^k}^2 - \frac{\mu_h}{2}\norms{x^{k+1} - x^k}^2.
\end{array}
}
Now, multiplying the above estimate by $1 - \tau_k \in [0, 1)$, and \eqref{eq:keylemma_Phi2} by $\tau_k \in (0, 1]$, and then summing up the results, 
and using the first two elementary expressions, we get
\myeq{eq:p2_keylemma_proof_sumtau}{
\hspace{-3ex}
\arraycolsep=0.2em
\begin{array}{lcl}
	\Lc_{\rho_k} (x^{k + 1}, s^{k + 1}, \tilde{y}^k) & \leq & (1 - \tau_k)\Lc_{\rho_k} (x^k, s^k, \tilde{y}^k)  + \tau_k \Lc_{\rho_k} (x, s, \tilde{y}^k) \vspace{1ex}\\
	&& + {~} L_k  \iprods{x^{k + 1} - \hat{x}^k, (1-\tau_k)x^k + \tau_k x - x^{k+1}}\vspace{1ex}\\
	& & - {~}  \frac{(1 - \tau_k)\rho_k}{2} \norms{[g(x^k) - s^k] - [g(\hat{x}^k) - s^{k + 1}]}^2\vspace{1ex}\\
	& & - {~}  \frac{\tau_k \rho_k}{2} \norms{[g(x) - s] - [g(\hat{x}^k) - s^{k + 1}]}^2 + \frac{\Lb_g^k + L_f + \rho_k M_g^2}{2} \norms{x^{k + 1} - \hat{x}^k}^2 \vspace{1ex}\\
	&& - {~} \frac{\mu_h\tau_k^2}{2}\norms{\frac{1}{\tau_k}[x^{k+1} - (1-\tau_k)x^k] - x}^2   -  \frac{(\mu_f + \mu_h)\tau_k(1-\tau_k)}{2}\norms{x^k - x}^2 \vspace{1ex}\\
	&& - {~} \frac{\mu_f\tau_k^2}{2}\norms{\frac{1}{\tau_k}[\hat{x}^k - (1-\tau_k)x^k] - x}^2.
\end{array}
\hspace{-2ex}
}
Next, by the definition of $\Lc_{\rho_k}$ in \eqref{eq:aug_Lag}, for any $y \in \dom{H^{*}}$, we have
\myeq{eq:lm_projz_Lrho}{
\arraycolsep=0.1em
\begin{array}{rl}
	\Lc_{\rho_k} (x^{k + 1}, s^{k + 1}, y) & - {~} (1 - \tau_k)\Lc_{\rho_k} (x^k, s^k, y)\vspace{1ex}\\
	= & \Lc_{\rho_k} (x^{k + 1}, s^{k + 1}, \tilde{y}^k) - (1 - \tau_k)\Lc_{\rho_k} (x^k, s^k, \tilde{y}^k)\vspace{1ex}\\
	& + {~}  \iprods{y - \tilde{y}^k,\ [g(x^{k + 1}) - s^{k + 1}] - (1 - \tau_k)[g(x^k) - s^k]} \hfill\quad (=: \Tc_1).
\end{array}
}
To analyze the term $\Tc_1$ in \eqref{eq:lm_projz_Lrho}, we denote $\Theta_k :=  g(x^{k}) - g(\hat{x}^{k-1}) + \frac{1}{\rho_{k-1}} (y^{k} - \tilde{y}^{k-1})$ and 
\myeq{eq:keylm_u}{
	u^{k + 1} := \eta_k \left([g(x^{k + 1}) - s^{k + 1}] - (1 - \tau_k)[g(x^k) - s^k]\right) = \eta_k [\Theta_{k + 1} - (1 - \tau_k)\Theta_k].
}
Therefore, $\tilde{y}^{k + 1}  =  \tilde{y}^k + u^{k + 1}$, and $\Tc_1$ becomes
\myeq{eq:keylm_z}{
\arraycolsep=0.2em
\begin{array}{lcl}
	\Tc_1 & \overset{\tiny\eqref{eq:keylm_u}}{=} & \frac{1}{\eta_k} \iprods{y - \tilde{y}^k,\ u^{k + 1}}  = 
	\frac{1}{2\eta_k} \left[ \norms{\tilde{y}^k - y}^2  - \norms{\tilde{y}^k + u^{k + 1} - y}^2 + \norms{u^{k + 1}}^2 \right] \vspace{1ex}\\
	& = & \frac{1}{2\eta_k} \left[ \norms{\tilde{y}^k - y}^2 - \norms{\tilde{y}^{k + 1} - y}^2 + \norms{u^{k + 1}}^2 \right].
\end{array}
}
Substituting \eqref{eq:keylm_z} into \eqref{eq:lm_projz_Lrho}, and then combining with \eqref{eq:p2_keylemma_proof_sumtau}, we can further derive
\myeq{eq:keylemma_almost_final_projz}{
\hspace{-3ex}
\arraycolsep=0.1em
\begin{array}{rl}
	\Lc_{\rho_k} & (x^{k + 1}, s^{k + 1}, y) \ \leq  \ (1 - \tau_k)\Lc_{\rho_k} (x^k, s^k, y)\hfill (:= \Tc_2)\vspace{1ex}\\
	& + {~}  \tau_k \Lc_{\rho_k} (x, s, \tilde{y}^k) -  \frac{\tau_k \rho_k}{2} \norms{[g(x) - s] - [g(\hat{x}^k) - s^{k + 1}]}^2\hfill (:= \Tc_3)\vspace{1ex}\\
	& + {~}  L_k\iprods{x^{k + 1} - \hat{x}^k, (1-\tau_k)x^k + \tau_k x - x^{k + 1}} + \frac{\Lb_g^k + L_f + \rho_k M_g^2}{2} \norms{x^{k + 1} - \hat{x}^k}^2\hfill\quad (:= \Tc_4)\vspace{1ex}\\
	& + {~}  \frac{1}{2\eta_k} \left[ \norms{\tilde{y}^k - y}^2 - \norms{\tilde{y}^{k + 1} - y}^2 \right] + \frac{1}{2\eta_k} \norms{u^{k + 1}}^2 \vspace{1ex}\\
	& - {~} \frac{(1 - \tau_k)\rho_k}{2} \norms{[g(x^k) - s^k] - [g(\hat{x}^k) - s^{k + 1}]}^2    -  \frac{(\mu_f + \mu_h)\tau_k(1-\tau_k)}{2}\norms{x^k - x}^2 \vspace{1ex}\\
	& - {~} \frac{\mu_h\tau_k^2}{2}\norms{\frac{1}{\tau_k}[x^{k+1} - (1-\tau_k)x^k] - x}^2 -  \frac{\mu_f\tau_k^2}{2}\norms{\frac{1}{\tau_k}[\hat{x}^k - (1-\tau_k)x^k] - x}^2.
\end{array}
}
We now estimate the terms $\Tc_2$, $\Tc_3$, and $\Tc_4$ in \eqref{eq:keylemma_almost_final_projz}. It is easy to see that
\myeq{eq:keylemma_T2}{
	\Tc_2 = (1 - \tau_k)\left[\Lc_{\rho_{k - 1}} (x^k, s^k, y) + \tfrac{(\rho_k - \rho_{k - 1})}{2} \norms{g(x^k) - s^k}\right].
}
By the definition of $\breve{y}^{k + 1} := (1-\tau_k)\breve{y}^k + \tau_k\big(\tilde{y}^k + \rho_k[g(\hat{x}^k) - s^{k+1}]\big)$ in \eqref{eq:y_average}, we have
\myeq{eq:keylm_use_ybar}{
	\Tc_3 = \Lc(x, s, \breve{y}^{k + 1}) - (1 - \tau_k)\Lc(x, s, \breve{y}^k) - \frac{\tau_k \rho_k}{2} \norms{g(\hat{x}^k) - s^{k + 1}}^2.
}
Using the relation $2\iprods{u,v} = \norms{u}^2 - \norms{u-v}^2 + \norms{v}^2$, we further have
\myeq{eq:keylemma_T4}{
\arraycolsep=0.2em
\begin{array}{lcl}
	\Tc_4 & = & \frac{L_k\tau_k^2}{2} \norms{\frac{1}{\tau_k}[\hat{x}^k - (1-\tau_k)x^k] - x}^2 -  \frac{L_k\tau_k^2}{2} \norms{\frac{1}{\tau_k}[x^{k + 1} - (1-\tau_k)x^k] - x}^2\vspace{1ex}\\
	&&  - {~} \frac{1}{2} \left(L_k - \Lb_g^k - L_f - \rho_k M_g^2\right)\norms{x^{k + 1} - \hat{x}^k}^2.
\end{array}
}
Substituting \eqref{eq:keylemma_T2}, \eqref{eq:keylm_use_ybar}, and \eqref{eq:keylemma_T4} into \eqref{eq:keylemma_almost_final_projz}, we get
\myeq{eq:keylemma_almost_final_dual1}{
\hspace{-5ex}
\arraycolsep=0.1em
\begin{array}{rl}
	\Lc_{\rho_k} (x^{k + 1}, & s^{k + 1}, y)  - {~} \Lc(x, s, \breve{y}^{k + 1}) \leq (1 - \tau_k)[\Lc_{\rho_{k - 1}} (x^k, s^k, y) - \Lc(x, s, \breve{y}^k)]\vspace{1ex}\\
	& + {~}  \frac{\tau_k^2}{2}(L_k-\mu_f)\norms{\frac{1}{\tau_k}[\hat{x}^k - (1-\tau_k)x^k] - x}^2  - \frac{(\mu_f + \mu_h)\tau_k(1-\tau_k)}{2}\norms{x^k - x}^2  \vspace{1ex}\\
	& - {~} \frac{\tau_k^2}{2}(L_k + \mu_h)\norms{\frac{1}{\tau_k}[x^{k + 1} - (1-\tau_k)x^k] - x}^2 +  \frac{1}{2\eta_k} \big[ \norms{\tilde{y}^k - y}^2 - \norms{\tilde{y}^{k + 1} - y}^2 \big]  \vspace{1ex}\\
	&  - {~}  \frac{1}{2} \left(L_k - \Lb_g^k - L_f - \rho_k M_g^2\right)\norms{x^{k + 1} - \hat{x}^k}^2 + \Tc_5, 
\end{array}
\hspace{-1ex}
}
where $\Tc_5$ is defined and upper bounded as follows:
\myeq{eq:keylemma_almost_final2_projz}{
\arraycolsep=0.2em
\begin{array}{lcl}
	\Tc_5 & :=  & \frac{1}{2\eta_k} \norms{u^{k + 1}}^2 + \frac{(1 - \tau_k)(\rho_k - \rho_{k - 1})}{2} \norms{g(x^k) - s^k}^2 - \frac{\tau_k \rho_k}{2} \norms{g(\hat{x}^k) - s^{k + 1}}^2\vspace{1ex}\\
	&& - {~}  \frac{(1 - \tau_k)\rho_k}{2} \norms{[g(x^k) - s^k] - [g(\hat{x}^k) - s^{k + 1}]}^2\vspace{1ex}\\
	& = & \frac{1}{2\eta_k} \norms{u^{k + 1}}^2 - \frac{\rho_k}2 \norms{[g(\hat{x}^k) - s^{k + 1}] - (1 - \tau_k)[g(x^k) - s^k]}^2 \vspace{1ex}\\
	&& - {~}  \frac{(1 - \tau_k)[\rho_{k - 1} - (1 - \tau_k)\rho_k]}2 \norms{g(x^k) - s^k}^2\vspace{1ex}\\
	& \leq & \frac{\rho_k \eta_k}{2(\rho_k - \eta_k)} \norms{g(x^{k + 1}) - g(\hat{x}^k)}^2 - \frac{(1 - \tau_k)[\rho_{k - 1} - (1 - \tau_k)\rho_k]}2 \norms{g(x^k) - s^k}^2 \vspace{1ex}\\
	& \overset{\tiny\eqref{eq:M_g}}{\leq} & \frac{\rho_k \eta_k M_g^2}{2(\rho_k - \eta_k)} \norms{x^{k + 1} - \hat{x}^k}^2 - \frac{(1 - \tau_k)[\rho_{k - 1} - (1 - \tau_k)\rho_k]}2 \norms{g(x^k) - s^k}^2.
\end{array}
}
Finally, substituting \eqref{eq:keylemma_almost_final2_projz} into \eqref{eq:keylemma_almost_final_dual1}, we eventually get \eqref{eq:p2_keylemma}.
\Eproof

\section{Ergodic Convergence}\label{sec:apdx:ergodic_rate}
This appendix proves Theorems~\ref{thm:BigO_ergodic_rate} and~\ref{thm:alg2_ergodic}.

\subsection{Technical Lemmas}\label{subsec:alg1_thm_ergodic}
Let us first bound $\sets{ \norms{\tilde{y}^k - y^{\star}}}$ and $\sets{\norms{x^k - x^{\star}}}$.

\begin{lemma}\label{le:bound_of_Lk}
Let $\sets{(x^k, \tilde{y}^k)}$ be generated by Algorithm~\ref{alg:A1}, where $L_k$, $\rho_k$, and $\eta_k$ satisfy \eqref{eq:non_erg_T_def2} and \eqref{eq:param_update_ergodic_O1}.
Then, for all $k \geq 0$, we have
\vspace{0.5ex}
\myeq{eq:non_erg_T_bd}{
	L_g \big[ \norms{y^{\star}} + \norms{\tilde{y}^k - y^{\star}} + \rho_k M_g \norms{x^k - x^{\star}} \big] \leq \rho_k C.
\vspace{0.5ex}
}
\end{lemma}

\begin{proof}
Denote $\Rc_0^2(x^{\star}, y^{\star}) := L_0\norms{x^0 - x^{\star}}^2 + \frac{1}{\eta_0}\norms{y^0 - y^{\star}}^2$.
Then, after a few elementary calculations, one can easily check that $C$, $L_0$, $\rho_0$, and $\eta_0$ chosen by \eqref{eq:non_erg_T_def2} satisfy
\myeq{eq:para_cond1}{
L_g \big[\norms{y^{\star}} + \big(\sqrt{\eta_0} + \tfrac{\rho_0 M_g}{\sqrt{L_0}}\big)\Rc_0(x^{\star}, y^{\star})\big] \leq \rho_0 C.
}
Now, we prove \eqref{eq:non_erg_T_bd} by induction. 
For $k = 0$, \eqref{eq:non_erg_T_bd} holds due to \eqref{eq:para_cond1}.

Suppose that \eqref{eq:non_erg_T_bd} holds for all $k \in \set{0,1,\cdots, K}$ with some $K \geq 0$, i.e., $L_g \big[ \norms{y^{\star}} + \norms{\tilde{y}^k - y^{\star}} + \rho_k M_g \norms{x^k - x^{\star}} \big] \leq \rho_k C$.
We now prove that \eqref{eq:non_erg_T_bd} also holds for $K + 1$. 
Indeed, using $y^{\star} = \prox_{\rho_k H^{*}} \left(y^{\star} + \rho_k g(x^{\star})\right)$ from \eqref{eq:opt_cond}, for $0 \leq k \leq K$ we have
\myeq{eq:ergodic_Lk_bd}{
\arraycolsep=0.2em
\begin{array}{lcl}
	\Lb_g^k & \stackrel{\eqref{eq:in_domH}}= & \Lb_g \left(\prox_{\rho_k H^{*}} \left(\tilde{y}^k + \rho_k g(x^k)\right)\right)\\
	& \stackrel{\eqref{eq:opt_cond}}= & \Lb_g \left(\prox_{\rho_k H^{*}} \left(\tilde{y}^k + \rho_k g(x^k)\right) - \prox_{\rho_k H^{*}}\left(y^{\star} + \rho_k g(x^{\star})\right) + y^{\star}\right)\vspace{1ex}\\
	& \leq & L_g \left[ \norms{\prox_{\rho_k H^{*}} \left(\tilde{y}^k + \rho_k g(x^k)\right) - \prox_{\rho_k H^{*}} \left(y^{\star} + \rho_k g(x^{\star})\right)} + \norms{y^{\star}}\right]\vspace{1ex}\\
	& \leq & L_g\left[ \norms{[\tilde{y}^k + \rho_k g(x^k)] - [y^{\star} + \rho_k g(x^{\star})]} + \norms{y^{\star}}\right]\vspace{1ex}\\
	& \leq & L_g \big[ \norms{y^{\star}} + \norms{\tilde{y}^k - y^{\star}} + \rho_k M_g \norms{x^k - x^{\star}} \big] \leq \rho_kC, 
\end{array}
}
where, in the third line, we applied Assumption \ref{as:A2}(b), in the fourth line we used the non-expansiveness of $\prox_{\rho_kH^{*}}(\cdot)$, and the last inequality is due to our induction assumption.

By the definition of $L_k$ and $\eta_k$ in \eqref{eq:non_erg_T_def2}, and using \eqref{eq:ergodic_Lk_bd}, for $0 \leq k \leq K$, we have 
\myeq{eq:non_erg_beta_term}{
	L_k - \Lb_g^k - L_f - \tfrac{{\rho_k^2} M_g^2}{\rho_k - \eta_k} \geq L_f + \rho_k(C + 2M_g^2) - \rho_k C - L_f - \tfrac{{\rho_k^2} M_g^2}{\rho_k/2}  = 0.
}
Using this estimate, we substitute $\mu_f = \mu_h := 0$, $\tau_k := 1$ and $\hat{x}^k = x^k$ (since $\beta_k = 0$) into \eqref{eq:p2_keylemma} of Lemma~\ref{lm:descent_property1} to obtain for any $(x, s, y)\in\dom{F}\times\dom{H}\times\dom{H^{*}}$ that
\myeq{eq:ergodic_induction}{
\arraycolsep=0.2em
\begin{array}{lcl}
	\Lc_{\rho}  (x^{k + 1}, s^{k + 1}, y) - \Lc(x, s, y^{k + 1})  & \leq & \frac{L}{2} \left[ \norms{x^k - x}^2 - \norms{x^{k + 1} - x}^2 \right] \vspace{1ex}\\
	&& + {~}  \frac{1}{2\eta} \left[ \norms{\tilde{y}^k - y}^2 - \norms{\tilde{y}^{k + 1} - y}^2 \right].
\end{array}
}
Here, we note that $L_k = L$ and $\eta_k = \eta$ for all $k\geq 0$.
By \eqref{eq:saddle_pt_s}, we have $\Lc_{\rho}  (x^{k + 1}, s^{k + 1}, y^{\star}) - \Lc(x^{\star}, s^{\star}, y^k) \geq 0$.
Hence, \eqref{eq:ergodic_induction} implies that
\myeqn{
	L \norms{x^{k + 1} - x^{\star}}^2 + \frac{1}{\eta} \norms{\tilde{y}^{k + 1} - y^{\star}}^2 \leq L\norms{x^k - x^{\star}}^2 + \frac{1}{\eta} \norms{\tilde{y}^k - y^{\star}}^2.
}
Since the above inequality holds for all $0 \leq k \leq K$, we can show by induction that
\myeqn{
\begin{array}{lcl}
	L \norms{x^{K + 1} - x^{\star}}^2  + \frac{1}{\eta} \norms{\tilde{y}^{K + 1} - y^{\star}}^2 & \leq & L \norms{x^K - x^{\star}}^2 + \frac1\eta \norms{\tilde{y}^K - y^{\star}}^2\vspace{1ex}\\
	& \leq & L \norms{x^0 - x^{\star}}^2 + \frac{1}{\eta} \norms{y^0 - y^{\star}}^2 = \Rc_0^2 (x^{\star}, y^{\star}).
\end{array}
}
The last inequality leads to $\norms{x^{K + 1} - x^{\star}} \leq \frac{1}{\sqrt{L}} \Rc_0(x^{\star}, y^{\star})$ and $\norms{\tilde{y}^{K + 1} - y^{\star}} \leq \sqrt{\eta} \Rc_0(x^{\star}, y^{\star})$. 
Using these bounds, \eqref{eq:para_cond1}, and $\rho_k = \rho_0 = \rho$, we can derive
\myeqn{
	L_g \big[ \norms{y^{\star}} + \norms{\tilde{y}^{K + 1} {\!\!\!} - y^{\star}} + \rho_k M_g \norms{x^{K + 1}  {\!\!\!} - x^{\star}}\big] 
	\leq L_g \big[\norms{y^{\star}} + \big(\sqrt{\eta_0} + \tfrac{\rho_0 M_g}{\sqrt{L_0}} \big)\Rc_0 (x^{\star}, y^{\star})\big] 
	\stackrel{\eqref{eq:para_cond1}}\leq \rho_0 C.
}
Hence, we prove that \eqref{eq:non_erg_T_bd} also holds for $K + 1$. 
By induction, it holds for all $k \geq 0$.
\end{proof}

\begin{lemma}\label{le:param_properties}
Let $\rho_k$, $L_k$, and $\eta_k$ be updated by \eqref{eq:scvx_ergodic_param_init} and \eqref{eq:scvx_ergodic_param_update}.
Then, for $k \geq 0$, we have
\myeq{eq:ergodic_O2_params_props}{
\left\{\begin{array}{l}
L_k \geq L_f + \rho_k(C + 2M_g^2), \quad 
\rho_k \geq \rho_0 + P_0 k, \quad 
0 < \theta_{k+1} \leq 1,\vspace{1ex}\\
L_k + \mu_h =  \frac{L_{k+1} - \mu_f}{\theta_{k + 1}},\quad \rho_k =  \theta_{k + 1} \rho_{k + 1},\quad \text{and} \quad \eta_k = \theta_{k + 1} \eta_{k + 1},
\end{array}\right.
}
where $P_0 := \frac{\rho_0}{2L_0}\big[ \sqrt{4L_0(\mu_f + \mu_h) + (2L_0 - \mu_f)^2} - (2L_0 - \mu_f) \big] > 0$.
\end{lemma}

\begin{proof}
The assertion $0 < \theta_{k+1} \leq 1$ obviously follows from the update of $\theta_{k+1}$ in \eqref{eq:scvx_ergodic_param_update}.
The last two equations of \eqref{eq:ergodic_O2_params_props} directly follow from the update of $\rho_k$ and $\eta_k$ in \eqref{eq:scvx_ergodic_param_update}.

Now, we prove  $L_k \geq L_f + \rho_k(C + 2M_g^2)$ of \eqref{eq:ergodic_O2_params_props} by induction.
First, it holds with equality for $k = 0$ due to the choice of $L_0$ in \eqref{eq:scvx_ergodic_param_init}. 
Furthermore, if it holds for all $k \in \sets{0, 1, \cdots, K}$ with some $K\geq 0$, then we can show that it holds for $K+1$.
Indeed, we have
\myeqn{
L_{K+1} \stackrel{\eqref{eq:scvx_ergodic_param_update}}= \frac{L_K}{\theta_{K+1}} \geq \frac{L_f}{\theta_{K+1}} + \frac{\rho_K}{\theta_{K+1}}(C + 2M_g^2) \geq L_f + \rho_{K+1}(C + 2M_g^2),
}
where the first inequality holds due to the induction assumption, while the second holds due to $0 < \theta_{K+1} \leq 1$ and $\rho_K =  \theta_{K + 1} \rho_{K + 1}$ (already proved).
Thus the first inequality of \eqref{eq:ergodic_O2_params_props} is also true for $K + 1$. 
By induction, the first inequality in \eqref{eq:ergodic_O2_params_props}  holds for all $k\geq 0$.

Next, to prove $\rho_k \geq \rho_0 + P_0 k$ in \eqref{eq:ergodic_O2_params_props} we notice that \eqref{eq:scvx_ergodic_param_update} implies
\myeqn{
\arraycolsep=0em
\begin{array}{lcl}
\rho_{k+1} & \stackrel{\eqref{eq:scvx_ergodic_param_update}}= & \rho_k\left(\frac{\mu_f + \sqrt{\mu_f^2 + 4L_k(L_k+\mu_h)}}{2L_k}\right) = \rho_k + \frac{\rho_k}{2L_k}\left(\mu_f + \sqrt{\mu_f^2 + 4L_k(L_k+\mu_h)} - 2L_k\right) \vspace{1ex}\\
&\geq & \rho_k + \frac{\rho_0}{2L_0}\big[\sqrt{(2L_0 - \mu_f)^2 + 4L_0(\mu_f + \mu_h)} - (2L_0 - \mu_f)\big]  = \rho_k + P_0,
\end{array}
}
where we have used $\frac{\rho_k}{L_k} = \frac{\rho_0}{L_0}$ from \eqref{eq:scvx_ergodic_param_update} and $L_k \geq L_0$.
By induction, we obtain $\rho_k \geq \rho_0 + P_0k$.

Finally, to prove the first statement in the second line of \eqref{eq:ergodic_O2_params_props}, from the update rule of $\theta_{k+1}$, we have $\theta_{k+1}^2(L_k + \mu_h) + \theta_{k+1}\mu_f = L_k$, which is equivalent to $L_k + \mu_h = \frac{1}{\theta_{k+1}}\big[\frac{L_k}{\theta_{k+1}} - \mu_f] = \frac{L_{k+1} - \mu_f}{\theta_{k+1}}$, where we have used $L_{k+1} = \frac{L_k}{\theta_{k+1}}$.
Hence, this assertion is proved.
\end{proof}

\begin{lemma}\label{le:bound_of_Lk_scvx}
Let $\rho_k$, $L_k$, and $\eta_k$ be updated by \eqref{eq:scvx_ergodic_param_init} and \eqref{eq:scvx_ergodic_param_update}.
Then, for $k \geq 0$, we have
\myeq{eq:ergodic_Tbd_str}{
	L_g \big[ \norms{y^{\star}} + \norms{\tilde{y}^k - y^{\star}} + \rho_k M_g\norms{x^k - x^{\star}} \big] \leq \rho_kC.
}
\end{lemma}

\begin{proof}
As shown in Lemma~\ref{le:bound_of_Lk}, if $\rho_0 := 1$ and we choose $C$ as in \eqref{eq:scvx_ergodic_param_init}, then we have
\myeq{eq:le:bound_of_Lk_scvx_proof0}{
L_g \big[\norms{y^{\star}} + \big(\sqrt{\eta_0} + \tfrac{\rho_0 M_g}{\sqrt{L_0 - \mu_f}}\big)\Rc^2_0(x^{\star}, y^{\star})\big] \leq \rho_0 C,
}
where $\Rc_0^2(x, y) := L_0\norms{x^0 - x}^2 + \frac{1}{\eta_0}\norms{y^0 - y}^2$.
We prove \eqref{eq:ergodic_Tbd_str} by induction. 
For $k = 0$, the inequality \eqref{eq:ergodic_Tbd_str} holds due to \eqref{eq:le:bound_of_Lk_scvx_proof0} and $L_0 > \mu_f$.
Suppose that \eqref{eq:ergodic_Tbd_str} holds for all $k \in \sets{0, 1, \cdots, K}$ with some $K \geq 0$.
Then, similar to  \eqref{eq:ergodic_Lk_bd}, for all $k\in \sets{0,1,\cdots, K}$, we have
\myeq{eq:ergodic_O2_Lk}{
	\Lb_g^k  \overset{\tiny\eqref{eq:ergodic_Lk_bd}}{\leq} L_g \left[ \norms{y^{\star}} + \norms{\tilde{y}^k - y^{\star}} + \rho_k M_g \norms{x^k - x^{\star}}\right] \leq \rho_k C.
}
By the update of $L_k$, $\rho_k$, and $\eta_k$ in \eqref{eq:scvx_ergodic_param_update}, the first inequality of \eqref{eq:ergodic_O2_params_props}, and \eqref{eq:ergodic_O2_Lk}, we have
\myeq{eq:ergodic_str_Lk_term}{
\arraycolsep=0.1em
\begin{array}{lcl}
	L_k - \Lb_g^k - L_f - \frac{\rho_k^2 M_g^2}{\rho_k - \eta_k} &\overset{\eqref{eq:scvx_ergodic_param_update},\eqref{eq:ergodic_O2_params_props}}{\geq} & 
	L_f + \rho_k(C + 2M_g^2) - L_f - \Lb_g^k - 2\rho_kM_g^2   \stackrel{\eqref{eq:ergodic_O2_Lk}}\geq  0.
\end{array}
}
Using this inequality,  $\tau_k := 1$, and $(x, s, y) := (x^{\star}, s^{\star}, y^{\star})$ into \eqref{eq:p2_keylemma} of Lemma~\ref{lm:descent_property1}, and noting that $\hat{x}^k = x^k$ (since $\beta_k = 0$) and $L_k - \mu_f > 0$ yields
\myeq{eq:ergodic_O2_telescope}{
\hspace{-1ex}
\arraycolsep=0.2em
\begin{array}{lcl}
	0 & \leq & \Lc_{\rho_k} (x^{k + 1}, s^{k + 1}, y^{\star}) - \Pc^{\star} \vspace{1ex}\\
	& \stackrel{\eqref{eq:ergodic_str_Lk_term}}\leq & \frac{(L_k - \mu_f)}{2} \norms{x^k - x^{\star}}^2 - \frac{(L_k + \mu_h)}{2} \norms{x^{k + 1} - x^{\star}}^2 + \frac{1}{2\eta_k} \left[ \norms{\tilde{y}^k - y^{\star}}^2 - \norms{\tilde{y}^{k + 1} - y^{\star}}^2 \right] \vspace{1ex} \\
	& \stackrel{\eqref{eq:ergodic_O2_params_props}}= & \left[ \frac{(L_k-\mu_f)}{2} \norms{x^k - x^{\star}}^2 + \frac{1}{2\eta_k} \norms{\tilde{y}^k - y^{\star}}^2\right] \vspace{1ex}\\
	&& - {~}  {~} \frac{1}{\theta_{k + 1}} \left[ \frac{(L_{k+1}-\mu_f)}{2} \norms{x^{k + 1} - x^{\star}}^2 + \frac{1}{2\eta_{k + 1}} \norms{\tilde{y}^{k + 1} - y^{\star}}^2\right].
\end{array}
\hspace{-2ex}
}
Multiplying \eqref{eq:ergodic_O2_telescope} by $2\rho_k$, and noticing that $\rho_{k+1} = \frac{\rho_k}{\theta_{k + 1}}$, we get
\myeqn{
\hspace{-1ex}
\arraycolsep=0.1em
\begin{array}{ll}
	\rho_{k+1} \big[(L_{k+1} \! - \mu_f) \norms{x^{k + 1} {\!\!\!} - x^{\star}}^2 + \frac{1}{\eta_{k + 1}} \norms{\tilde{y}^{k + 1} {\!\!\!} - y^{\star}}^2\big]  \leq  \rho_{k}[(L_{k} - \mu_f) \norms{x^{k} - x^{\star}}^2  + \frac{1}{\eta_{k}} \norms{\tilde{y}^{k } - y^{\star}}^2].
\end{array}
\hspace{-2ex}
}
By induction, the above inequality holds for $k \in \sets{0, 1, \cdots, K}$.
Consequently, one has
\myeqn{
\hspace{-0.5ex}
\arraycolsep=0.1em
\begin{array}{lcl}
	(L_{K + 1} -\mu_f)\norms{x^{K + 1} {\!\!} - x^{\star}}^2 + \tfrac{1}{\eta_{K + 1}} \norms{\tilde{y}^{K + 1} {\!\!} - y^{\star}}^2 
	& \leq & \frac{\rho_0}{\rho_{K+1}}\left[ (L_0-\mu_f)\norms{x^0 {\!} -  x^{\star}}^2 + \tfrac{1}{\eta_0} \norms{y^0 {\!} - y^{\star}}^2 \right] \vspace{1ex}\\
	& \leq & \Rc_0^2 (x^{\star}, y^{\star}) \quad \text{(since $\frac{\rho_0}{\rho_{K+1}} \leq 1$)},
\end{array}	
\hspace{-2ex}	
\vspace{-1ex}
}
which implies that $\norms{x^{K + 1} - x^{\star}} \leq \frac{\Rc_0(x^{\star}, y^{\star})}{\sqrt{L_{K+1}-\mu_f}}$ and $\norms{\tilde{y}^{K + 1} - y^{\star}} \leq \sqrt{\eta_{K + 1}} \Rc_0(x^{\star}, y^{\star})$. 

\noindent
Finally, using the above estimates, we can easily deduce
\myeqn{
\arraycolsep=0.1em
\begin{array}{ll}
& \frac{1}{\rho_{K + 1}} \big[ \norms{y^{\star}} + \norms{\tilde{y}^{K + 1} {\!\!} - y^{\star}} + \rho_{K + 1} M_g \norms{x^{K + 1} {\!\!} - x^{\star}} \big] \leq  \frac{\norms{y^{\star}}}{\rho_{K + 1}}  + \Big[ \frac{\sqrt{\eta_{K+1}}}{\rho_{K + 1}} + \frac{M_g}{\sqrt{L_{K+1} - \mu_f}} \Big] \Rc_0(x^{\star}, y^{\star})\\
&	\stackrel{\eqref{eq:scvx_ergodic_param_update}}\leq  \frac1{\rho_0} \norms{y^{\star}} + \Big[ \frac{\sqrt{\eta_0}}{\rho_0} + \frac{M_g}{\sqrt{L_0 - \mu_f}} \Big] \Rc_0(x^{\star}, y^{\star}) \overset{\tiny\eqref{eq:le:bound_of_Lk_scvx_proof0}}{\leq} C/L_g.
\end{array}
\vspace{-0.5ex}
}
This inequality shows that \eqref{eq:ergodic_Tbd_str} also holds for $K + 1$. 
By induction, we have thus proved that \eqref{eq:ergodic_Tbd_str} holds for all $k \geq 0$.
\end{proof}

\subsection{The proof of Theorem~\ref{thm:BigO_ergodic_rate}}
First, we use \eqref{eq:p2_keylemma} of Lemma~\ref{lm:descent_property1} and \eqref{eq:non_erg_T_bd} of Lemma \ref{le:bound_of_Lk}, and follow the same lines of proof as \eqref{eq:ergodic_Lk_bd} and \eqref{eq:non_erg_beta_term} to show that $L - L_f - \Lb_g^k - \tfrac{{\rho^2} M_g^2}{\rho - \eta} \geq 0$.
Therefore, similar to \eqref{eq:ergodic_induction}, for any $y \in \dom{H^{*}}$ and any $j \geq 0$, we have
\myeqn{
\hspace{-0.5ex}
	\Lc(x^{j + 1}, s^{j + 1}, y) - \Lc(x, s, y^{j + 1}) \leq  \tfrac{L}{2} \left[ \norms{x^j - x}^2 - \norms{x^{j + 1} {\!\!} - x}^2 \right] +  \tfrac{1}{2\eta} \left[ \norms{\tilde{y}^j - y}^2 - \norms{\tilde{y}^{j + 1}  {\!\!} - y}^2 \right].
\hspace{-1ex}
}
Summing up this inequality from $j := 0$ to $j := k - 1$, we get
\myeqn{
\vspace{-1ex}
	\sum_{j = 0}^{k - 1} \left[\Lc(x^{j + 1}, s^{j + 1}, y) -  \Lc(x, s, y^{j + 1})\right] \leq \frac{1}{2} \left[ L \norms{x^0 - x}^2 + \frac1\eta \norms{y^0 - y}^2\right] = \frac{\Rc_0^2 (x,y)}2.
\vspace{-1ex}
}
Dividing the above inequality by $k \geq 1$, and using the convexity of $\Lc$ in $x$ and $s$, and its concavity in $y$, the definition of $\bar{x}^k$ and $\bar{y}^k$, and $\bar{s}^k := \tfrac1k \sum_{j = 1}^{k} s^j$, we get
\vspace{-1ex}
\myeq{eq:ergodic_thm_L}{
	\Lc(\bar{x}^k, \bar{s}^k, y) - \Lc(x, s, \bar{y}^k)  \leq  \frac{1}{k}\sum_{j = 1}^k [\Lc(x^j, s^j, y) - \Lc(x, s, y^j)] \leq \frac{\Rc_0^2 (x,y)}{2k}.
\vspace{-1ex}	
}
By \eqref{eq:L_to_Ltilde}, we have $\widetilde{\Lc}(\bar{x}^k, y) \leq \Lc(\bar{x}^k, \bar{s}^k, y)$ and $\widetilde{\Lc}(x, \bar{y}^k) = \Lc(x, \breve{s}^k, \bar{y}^k)$ for $\breve{s}^k\in \partial{H}^{*}(\bar{y}^k)$.
Hence, $\widetilde{\Lc}(\bar{x}^k, y) - \widetilde{\Lc}(x, \bar{y}^k) \leq \Lc(\bar{x}^k, \bar{s}^k, y) - \Lc(x, \breve{s}^k, \bar{y}^k)$.
Substituting $s := \breve{s}^k$ and this inequality into \eqref{eq:ergodic_thm_L}, we obtain $\widetilde{\Lc}(\bar{x}^k, y) - \widetilde{\Lc}(x, \bar{y}^k)  \leq \tfrac{\Rc_0^2 (x,y)}{2k}$.
Taking the supremum on both sides of this estimate over $\Xc \times \Yc$ and using $\Gc_{\Xc \times \Yc}$ from \eqref{eq:gap_func}, we prove the third assertion of \eqref{eq:gap_convergence_ergodic}.

Next, if $H$ is $M_H$-Lipschitz continuous, then we let $\breve{y}^k := \tfrac{M_H}{\norms{g(\bar{x}^k) - \bar{s}^k}} [g(\bar{x}^k) - \bar{s}^k]$, 
and substitute $(x, s, y) := (x^{\star}, s^{\star}, \breve{y}^k)$ in \eqref{eq:ergodic_thm_L} to get
\myeqn{
\arraycolsep=0.1em
\begin{array}{lcl}
	\Pc(\bar{x}^k) - \Pc^{\star} & \overset{\tiny\eqref{eq:primal_cvx}}{=} & F(\bar{x}^k) + H(g(\bar{x}^k)) - \Pc^{\star} \leq F(\bar{x}^k) + H(\bar{s}^k) + |H(g(\bar{x}^k)) - H(\bar{s}^k)| - \Pc^{\star} \vspace{1ex}\\
	& \leq & F(\bar{x}^k) + H(\bar{s}^k) + M_H\norms{ g(\bar{x}^k) - \bar{s}^k} - \Pc^{\star}\\
	& = & F(\bar{x}^k) + H(\bar{s}^k) + \iprods{\breve{y}^k, g(\bar{x}^k) - \bar{s}^k} - \Pc^{\star} \leq \Lc(\bar{x}^k, \bar{s}^k, \breve{y}^k) - \Pc^{\star} \stackrel{\eqref{eq:ergodic_thm_L}}\leq \frac{\Rc_0^2 (x^{\star}, \breve{y}^k)}{2k}.
\end{array}
}
Using $\norms{y^0 - \breve{y}^k}^2 \leq (\norms{y^0} + \norms{\breve{y}^k})^2 = {(\norms{y^0} + M_H)}^2$ to upper bound $\Rc_0^2 (x^{\star}, \breve{y}^k)$ in the last estimate, we obtain the first assertion of \eqref{eq:gap_convergence_ergodic}.

Let $\breve{x}^k$ satisfy $0 \in  g'(\breve{x}^k)^\top \bar{y}^k + \partial F(\breve{x}^k)$, then by the form of \eqref{eq:dual_cvx}, we have $\Dc(\bar{y}^k) = \widetilde{\Lc}(\breve{x}^k, \bar{y}^k) = \Lc(\breve{x}^k,\breve{s}^k,\bar{y}^k)$ for $\breve{s}^k \in \partial H^{*} (\bar{y}^k)$.
Moreover, notice that $-\Dc^{\star} =  \Lc(x^{\star}, s^{\star}, y^{\star}) \leq \Lc(\bar{x}^k,\bar{s}^k, y^{\star})$ by \eqref{eq:saddle_pt_s}.
Therefore, substituting $(x, s, y) := (\breve{x}^k, \breve{s}^k, y^{\star})$ into \eqref{eq:ergodic_thm_L}, we can derive
\myeqn{
	\Dc(\bar{y}^k) - \Dc^{\star}  \leq \Lc(\bar{x}^k, \bar{s}^k, y^{\star}) - \Lc(\breve{x}^k, \breve{s}^k, \bar{y}^k) \stackrel{\eqref{eq:ergodic_thm_L}}\leq \frac{\Rc_0^2 (\breve{x}^k, y^{\star})}{2k}.
}
Since $0 \in  g'(\breve{x}^k)^\top \bar{y}^k + \partial F(\breve{x}^k)$, we have $\breve{x}^k \in \partial{F}^{*}(-g^{\prime} (\breve{x}^k)^\top \bar{y}^k)$.
If $F^{*}$ is $M_{F^{*}}$-Lipschitz continuous, then $\norms{\breve{x}^k} = \norms{\nabla{F}^{*}(-g^{\prime} (\breve{x}^k)^\top \bar{y}^k)} \leq M_{F^{*}}$, thus $\norms{x^0 - \breve{x}^k}^2 \leq {(\norms{x^0} + M_{F^{*}})}^2$. 
Substituting this into $\Rc_0^2 (\breve{x}^k, y^{\star})$ of the last inequality, we get the second line of \eqref{eq:gap_convergence_ergodic}.
\Eproof

\subsection{The proof of Theorem \ref{thm:alg2_ergodic}}\label{subsec:alg2_ergodic_proof}
By Lemma \ref{le:bound_of_Lk_scvx}, \eqref{eq:ergodic_Tbd_str} holds for all $k \geq 0$. 
Thus using the same lines of proof from \eqref{eq:ergodic_O2_Lk} to \eqref{eq:ergodic_O2_telescope}, we have for all $j \geq 0$ and any $(x, s, y) \in \dom{F}\times\dom{H}\times\dom{H^{*}}$ that
\myeqn{
\arraycolsep=0.1em
\begin{array}{rl}
	&\Lc(x^{j + 1}, s^{j + 1}, y)  -  \Lc(x, s, y^{j + 1}) \stackrel{\tau_j = 1}= \Lc(x^{j + 1}, s^{j + 1}, y) - \Lc(x, s, \breve{y}^{j + 1})\vspace{1ex}\\
	&\ \ \leq \left[ \frac{(L_j-\mu_f)}{2} \norms{ x^j - x}^2 + \frac{1}{2\eta_j} \norms{\tilde{y}^j - y}^2\right] - \frac{1}{\theta_{j + 1}} \left[ \frac{(L_{j+1}-\mu_f)}{2} \norms{x^{j + 1} {\!\!} - x}^2 + \frac{1}{2\eta_{j + 1}} \norms{\tilde{y}^{j + 1} {\!\!} -  y}^2\right].
\end{array}
}
Multiplying the last inequality by $2\rho_j$ and noticing that $\tfrac{\rho_j}{\theta_{j + 1}} = \rho_{j + 1}$, it leads to
\myeqn{
\hspace{-1ex}
\arraycolsep=0.2em
\begin{array}{lcl}
	2\rho_j [\Lc(x^{j + 1}, s^{j + 1}, y) - \Lc(x, s, y^{j + 1})] & \leq & \rho_j \left[ (L_j-\mu_f) \norms{ x^j - x}^2 + \frac{1}{\eta_j} \norms{\tilde{y}^j - y}^2\right]\vspace{1ex}\\
	& & - {~} \rho_{j + 1} \left[ (L_{j + 1} -\mu_f) \norms{ x^{j + 1} {\!\!} - x}^2 + \frac{1}{\eta_{j + 1}} \norms{\tilde{y}^{j + 1} {\!\!} - y}^2\right].
\end{array}
\hspace{-1ex}
}
Summing up this inequality from $j := 0$ to $j := k - 1$, we obtain
\myeqn{
	\sum_{j = 0}^{k - 1} \rho_j [\Lc(x^{j + 1}, s^{j + 1}, y) - \Lc(x, s, y^{j + 1})] \leq  \frac{\rho_0}{2}\left[ L_0 \norms{x^0 - x}^2 + \frac{1}{\eta_0} \norms{y^0 - y}^2\right] = \frac{\rho_0 \Rc_0^2 (x, y)}{2},
}
where $\Rc_0^2 (x, y) := L_0 \norms{x^0 - x}^2 + \frac{2}{\rho_0} \norms{y^0 - y}^2$.
Dividing it by $\sum_{j = 0}^{k - 1} \rho_j$, and using the convexity of $\Lc$ in $x$ and $s$, and the concavity in $y$, and $\sets{(\bar{x}^k, \bar{y}^k)}$ defined by \eqref{eq:ergodic_O2_xbar_sbar} and $\bar{s}^k := \big(\sum_{j = 0}^{k-1} \rho_j\big)^{-1}\sum_{j = 0}^{k-1} \rho_j s^{j+1}$, we get
\myeqn{
\arraycolsep=0.2em
\begin{array}{lcl}
	\Lc(\bar{x}^k, \bar{s}^k, y) - \Lc(x, s, \bar{y}^k) & \overset{\tiny\eqref{eq:ergodic_O2_xbar_sbar}}{\leq} & \frac{1}{\sum_{j = 0}^{k - 1} \rho_j} \sum_{j = 0}^{k - 1} \rho_j \left[\Lc(x^{j + 1}, s^{j + 1}, y) - \Lc(x, s, y^{j + 1})\right]  \leq   \frac{\rho_0 \Rc_0^2 (x, y)}{2\sum_{j = 0}^{k - 1} \rho_j}.
\end{array}
}
By the second inequality in \eqref{eq:ergodic_O2_params_props} of Lemma \ref{le:param_properties}, we have  $\sum_{j=0}^{k-1}\rho_j \geq  k\rho_0 + \frac{1}{2}P_0 k(k - 1)$.
Combining the above two inequalities, we eventually get
\myeqn{
	\Lc(\bar{x}^k, \bar{s}^k, y) - \Lc(x, s, \bar{y}^k) \leq \frac{\Rc_0^2 (x, y)}{2 \rho_0k + P_0k(k-1)}.
}
Therefore, we can use the same arguments as in the proof of Theorem \ref{thm:BigO_ergodic_rate} to prove \eqref{eq:ergodic_O2_result}.
We omit repeating this derivation here.
\Eproof

\section{Semi-Ergodic Convergence}\label{sec:apdx:semi_ergodic_convg}
We now prove Theorems~\ref{thm:BigO_nonegrodic_rate} and \ref{thm:BigO_nonegrodic_rate_str}.
The following lemma will be used to prove Theorem~\ref{thm:BigO_nonegrodic_rate_str}.

\begin{lemma}\label{le:element_est12}
Given $\mu_f \geq 0$, $\mu_h \geq 0$, $L_k \geq \mu_f$, and $\tau_k \in (0, 1)$, let $m_k := \frac{L_k + \mu_h}{L_{k-1} + \mu_h}$ and $\sets{x^k}$ be a given sequence in $\R^p$.
We define $\hat{x}^k := x^k + \beta_k(x^k - x^{k-1})$, where $\beta_k := \frac{(1-\tau_{k-1})\tau_{k-1}}{\tau_{k-1}^2 + m_k\tau_k}$.
Assume that the following two conditions hold:
\myeq{eq:element_cond2}{
\arraycolsep=0.2em
\left\{\begin{array}{llcl}
&(L_{k-1} + \mu_h)(1-\tau_k)\tau_{k-1}^2 + (\mu_f + \mu_h)(1-\tau_k)\tau_k  & \geq & (L_k - \mu_f)\tau_k^2, \vspace{1ex}\\
&(L_{k-1} + \mu_h)(\tau_{k-1}^2 + m_k\tau_k)m_k\tau_k & \geq & (L_k - \mu_f)\tau_{k-1}^2.
\end{array}\right.
}
Then, for any $x \in \R^p$, we have
\myeq{eq:element_est12}{
\hspace{-2ex}
\arraycolsep=0.2em
\begin{array}{ll}
&(L_k - \mu_f)\tau_k^2\norms{\frac{1}{\tau_k}[\hat{x}^k - (1-\tau_k)x^k] - x}^2 -  (\mu_f + \mu_h) \tau_k(1-\tau_k)\norms{x^k - x}^2 \vspace{1ex}\\
&\qquad \leq {~} (1-\tau_k) \left(L_{k-1} + \mu_h\right)\tau_{k-1}^2\norms{\frac{1}{\tau_{k-1}}[x^k - (1-\tau_{k-1})x^{k-1}] - x}^2.
\end{array}
\hspace{-2ex}
}
\end{lemma}

\begin{proof}
Let $\omega_k := \frac{1}{\beta_k}$.
Then,  we have $\omega_k(\hat{x}^k - x^k) = x^k - x^{k-1}$.
Using this expression, we can easily show that \eqref{eq:element_est12} is equivalent to
\myeq{eq:key_eq_needs_proof}{
-2c_1\iprods{\hat{x}^k - x^k, x^k - x} \leq c_2\norms{\hat{x}^k - x^k}^2 + c_3\norms{x^k - x}^2,
}
where $c_1$, $c_2$, and $c_3$ are given as
\myeqn{
\arraycolsep=0.2em
\left\{\begin{array}{lcl}
c_1 &:= & \left(L_{k-1}  + \mu_h\right)(1-\tau_k)(1-\tau_{k-1})\tau_{k-1}\omega_k - (L_k - \mu_f)\tau_k,  \vspace{1ex}\\
c_2 &:= & \left(L_{k-1} + \mu_h\right)(1-\tau_k)(1-\tau_{k-1})^2\omega_k^2 - L_k + \mu_f,  \vspace{1ex}\\
c_3 &:= & \left(L_{k-1} + \mu_h\right)(1-\tau_k)\tau_{k-1}^2 - (L_k - \mu_f)\tau_k^2 + (\mu_f + \mu_h)(1-\tau_k)\tau_k.
\end{array}\right.
}
If the first condition of \eqref{eq:element_cond2} holds, then $c_3 \geq 0$.
Now, since $\omega_k = \frac{1}{\beta_k} = \frac{\tau_{k-1}^2 + m_k\tau_k}{(1-\tau_{k-1})\tau_{k-1}}$, we can easily show that $c_1 = c_3 \geq 0$.
Similarly, using the second condition of \eqref{eq:element_cond2}, we can also show that $c_2 \geq c_1 \geq 0$.
Hence, the inequality \eqref{eq:key_eq_needs_proof} automatically holds, which proves \eqref{eq:element_est12}.
\end{proof}

\subsection{The proof of Theorem~\ref{thm:BigO_nonegrodic_rate}}\label{subsec:alg1_thm_non_ergodic}
Since $H$ is $M_H$-Lipschitz continuous, one can easily show that $\norms{\prox_{\rho_k H^{*}} \left(\tilde{y}^k + \rho_k g(\hat{x}^k)\right)} \leq M_H$, see, e.g., \cite{TranDinh2015b}.
Hence, we obtain
\myeq{eq:bound_Lk_ratio}{
\hspace{-2ex}
\arraycolsep=0.2em
\begin{array}{lcl}
	\Lb_g^k & =  & \Lb_g (y^{k + 1}) \stackrel{\eqref{eq:in_domH}}\leq L_g \norms{\prox_{\rho_k H^{*}} \left(\tilde{y}^k + \rho_k g(\hat{x}^k)\right)} \leq L_gM_H. 
\end{array}
\hspace{-1ex}
}
Therefore, by the update rule of $L_k$ and $\eta_k$ in \eqref{eq:update_rule}, and \eqref{eq:bound_Lk_ratio}, we can easily show that
\myeqn{
\hspace{-2ex}
\arraycolsep=0.1em
\begin{array}{lcl}
	L_k  - {~} \Lb_g^k - L_f - \frac{\rho_k^2 M_g^2}{\rho_k - \eta_k} &\geq & L_k - L_gM_H - L_f - \frac{M_g^2\rho_k}{\gamma} = 0.
\end{array}
\hspace{-2ex}
}
Furthermore, other conditions in \eqref{eq:update_rule} ensure that
\myeqn{
	\rho_k > \eta_k,\quad \frac{1}{2\eta_k} = \frac{1 - \tau_k}{2\eta_{k - 1}},\quad \rho_{k - 1} - (1 - \tau_k)\rho_k = 0,\quad \text{and}\quad  L_k\tau_k^2  \leq (1 - \tau_k)L_{k - 1}\tau_{k - 1}^2.
}
Utilizing these relations, and denoting $\tilde{x}^k := \frac{1}{\tau_k}[\hat{x}^k - (1-\tau_k)x^k]$ and $\tilde{x}^{k+1} := \frac{1}{\tau_k}[x^{k+1} - (1-\tau_k)x^k]$, we can simplify the estimate \eqref{eq:p2_keylemma} of Lemma~\ref{lm:descent_property1} to get
\myeqn{
\arraycolsep=0.1em
\begin{array}{rl}
	\Lc_{\rho_k}\!& (x^{k + 1}, s^{k + 1}, y) - \Lc(x, s, \breve{y}^{k + 1}) + \frac{\tau_k^2L_k}{2} \norms{\tilde{x}^{k + 1} - x}^2 + \frac{1}{2\eta_k} \norms{\tilde{y}^{k + 1} - y}^2\vspace{1ex}\\
	\leq & (1 - \tau_k)\Big[\Lc_{\rho_{k - 1}} (x^k, s^k, y) - \Lc(x, s, \breve{y}^k) + \frac{\tau_{k - 1}^2L_{k - 1}}{2} \norms{\tilde{x}^k - x}^2 + \frac{1}{2\eta_{k - 1}} \norms{\tilde{y}^k - y}^2 \Big],
\end{array}
}
for $(x, s, y) \in \dom{F}\times\dom{H}\times\dom{H^{*}}$.
By induction, this inequality implies that
\myeqn{
\arraycolsep=0.1em
\begin{array}{rl}
	\Lc(x^k, s^k, y) & - {~} \Lc(x, s, \breve{y}^k) \leq \Lc_{\rho_{k - 1}} (x^k, s^k, y) - \Lc(x, s, \breve{y}^k)\vspace{1ex}\\
	\leq~~~ & \left[\prod_{j = 1}^{k - 1} (1 - \tau_j)\right]\Big[ \Lc_{\rho_0} (x^1, s^1, y^{\star}) - \Lc(x, s, \breve{y}^1) + \frac{L_0\tau_0^2}{2} \norms{\tilde{x}^1 - x}^2 + \frac{1}{2\eta_0} \norms{\tilde{y}^1 - y}^2 \Big]\\
	\stackrel{\eqref{eq:p2_keylemma},\eqref{eq:update_rule}}\leq & \frac{1}{k} \Big[ (1 - \tau_0)\big[ \Lc_{\rho_0} (x^0, s^0, y^{\star}) - \Lc(x, s, y^0) \big] + \frac{L_0\tau_0^2}{2} \norms{\tilde{x}^0 - x}^2 + \frac{1}{2\eta_0} \norms{y^0 - y}^2 \Big] \vspace{1ex}\\
	\stackrel{\tau_0=1}=~ & \frac{1}{2k} \left[ L_0 \norms{x^0 - x}^2 + \frac{1}{(1-\gamma)\rho_0} \norms{y^0 - y}^2\right] = \frac{\Rc_0^2 (x, y)}{2k},
\end{array}
}
where $\Rc_0^2(x,y) := L_0\norms{x^0 - x}^2 + \frac{1}{(1-\gamma)\rho_0}\norms{y^0 - y}^2$.
Taking $\breve{s}^k \in \partial{H^{*}}(y^k)$, and then using the same arguments as in the proof of \eqref{eq:gap_convergence_ergodic}, we can also show that
\myeqn{
	\widetilde{\Lc}(x^k, y) - \widetilde{\Lc}(x, \bar{y}^k) \leq \Lc(x^k, s^k, y) - \Lc(x, \breve{s}^k, \breve{y}^k) \leq \frac{\Rc_0^2 (x, y)}{2 k}.
}
The rest of the proof of Theorem \ref{thm:BigO_nonegrodic_rate} is similar to the lines after \eqref{eq:ergodic_thm_L} in the proof of Theorem \ref{thm:BigO_ergodic_rate}, except that we replace $\bar{x}^k$ there by $x^k$. 
Thus we omit the verbatim here.

Finally, since $\tilde{x}^k := \frac{1}{\tau_k}[\hat{x}^k - (1-\tau_k)x^k]$ and $\tilde{x}^{k+1} := \frac{1}{\tau_k}[x^{k+1} - (1-\tau_k)x^k]$, we obtain $\hat{x}^{k+1} = x^{k+1} + \beta_{k+1}(x^{k+1} - x^k)$ as in \eqref{eq:APD_scheme1}, where $\beta_{k+1} := \frac{(1-\tau_k)\tau_{k+1}}{\tau_k}$.
\Eproof

\subsection{The proof of Theorem \ref{thm:BigO_nonegrodic_rate_str}}\label{subsec:alg2_non_ergodic_proof}
First of all, the update of $\set{\tau_k}$ in \eqref{eq:update_rule_str} leads to
\myeq{eq:tau_str_properties}{
	\tau_k^2 = (1 - \tau_k)\tau_{k - 1}^2\quad \text{and}\quad \frac{1}{k + 1} \leq \tau_k \leq \frac2{k + 2}.
}
Thus $\rho_{k - 1} = \tfrac{\rho_0}{\tau_{k - 1}^2} = \tfrac{(1 - \tau_k)\rho_0}{\tau_k^2} = (1 - \tau_k)\rho_k$, which implies
\myeq{eq:param_str_prop_others}{
	\rho_k > \eta_k,\quad \rho_k \geq \rho_0,\quad \frac{1}{2\eta_k} = \frac{1 - \tau_k}{2\eta_{k - 1}},\quad \text{and}\quad \rho_{k - 1} - (1 - \tau_k)\rho_k = 0.
}
In addition, the condition \eqref{eq:element_cond2} holds if $0 < \rho_0 \leq \frac{\mu_F}{L_gM_H + M_g^2}$.
Moreover, since $\beta_{k+1}$ is updated by \eqref{eq:update_rule_str}, it also satisfies the condition of Lemma~\ref{le:element_est12}, leading to the update $x^{k+1} := x^{k+1} + \beta_{k+1}(x^{k+1} - x^k)$ in \eqref{eq:APD_scheme1}.
Hence, \eqref{eq:element_est12} of Lemma \ref{le:element_est12} holds.

Now, substituting \eqref{eq:element_est12} into \eqref{eq:p2_keylemma} of Lemma~\ref{lm:descent_property1} and using $L_k -  \Lb_g^k - L_f - \frac{\rho_k^2 M_g^2}{\rho_k - \eta_k} \geq 0$ as shown in the proof of Theorem~\ref{thm:BigO_nonegrodic_rate} above, and \eqref{eq:param_str_prop_others}, we get
\myeq{eq:p2_keylemma_scvx2}{
\arraycolsep=0.3em
\begin{array}{rl}
	\Lc_{\rho_k} (x^{k + 1}, s^{k + 1}, y) & - {~} \Lc(x, s, \breve{y}^{k + 1}) \ \leq \ (1 - \tau_k)\big[ \Lc_{\rho_{k - 1}} (x^k, s^k, y) - \Lc(x, s, \breve{y}^k) \big]\vspace{1ex}\\
	& + {~} (1-\tau_k)\frac{\tau_{k-1}^2}{2} \big(L_{k-1} + \mu_h\big) \norms{\frac{1}{\tau_{k-1}}[x^k - (1-\tau_{k-1})x^{k-1}] - x}^2 \vspace{1ex}\\
	& -  {~} \frac{\tau_k^2}{2}\big(L_k + \mu_h\big) \norms{\frac{1}{\tau_k}[x^{k + 1} - (1-\tau_k)x^k] - x}^2 \vspace{1ex}\\
	& + {~} \frac{(1-\tau_k)}{2\eta_{k-1}}\norms{\tilde{y}^k - y}^2 - \frac{1}{2\eta_k}\norms{\tilde{y}^{k + 1} - y}^2.
\end{array}
}
If we define the following potential function
\myeqn{
\begin{array}{lcl}
\Vc_k(x, y) & := & \Lc_{\rho_{k - 1}} (x^k, s^k, y) - \Lc(x, s, \breve{y}^k) + \frac{1}{2\eta_{k-1}}\norms{\tilde{y}^k - y}^2 \vspace{1ex}\\
&& + {~} \frac{\tau_{k-1}^2}{2} \big(L_{k-1} + \mu_h\big) \norms{\frac{1}{\tau_{k-1}}[x^k - (1-\tau_{k-1})x^{k-1}] - x}^2,
\end{array}
}
then \eqref{eq:p2_keylemma_scvx2} is equivalent to $\Vc_{k+1}(x, y) \leq (1-\tau_k)\Vc_k(x, y)$.
By induction, we obtain 
\myeqn{
\begin{array}{lcl}
\Lc(x^{k}, s^{k}, y) - \Lc(x, s, \breve{y}^k) &\leq & \Vc_k(x, y) \leq \Big[ \prod_{i=1}^{k-1}(1-\tau_i)\Big]\Vc_1(x, y) \vspace{1ex}\\
&\leq & \frac{4}{(k+1)^2}\left[ \frac{(L_0-\mu_f)}{2}\norms{x^0 - x}^2 + \frac{1}{2\eta_0}\norms{y^0 - y}^2 \right].
\end{array}
}
where we have used \eqref{eq:p2_keylemma} once again and $\prod_{i=1}^{k-1}(1-\tau_i) = \frac{\tau_{k-1}^2}{\tau_0^2} \leq \frac{4}{(k+1)^2}$ in the last inequality.
Using the last estimate and $0 < L_0 - \mu_f \leq L_0$, we can prove \eqref{eq:BigO_rate_str} in a similar manner as in the proof of Theorem \ref{thm:BigO_nonegrodic_rate}.
We therefore omit the details here.
\Eproof

\section{The Proof of Theorem \ref{thm:conic}}\label{app:conic}
Since $H(\cdot) = \delta_{-\Kc}(\cdot)$, which is not $M_H$-Lipschitz continuous, to prove assertions (c) and (d), we first bound $\Lb^k_g$ of Lemma~\ref{lm:descent_property1} as in \eqref{eq:bound_Lk_ratio1} below to update $L_k$ as in Theorem~\ref{thm:conic}.
Since $\tilde{y}^{k + 1} := \tilde{y}^k + \eta_k [\Theta_{k + 1} - (1 - \tau_k)\Theta_k]$ and $\eta_{k-1} = (1-\tau_k)\eta_k$ due to \eqref{eq:update_rule} or \eqref{eq:update_rule_str}, we have
$\tilde{y}^{k + 1} - \eta_k \Theta_{k + 1} = \tilde{y}^k - (1 - \tau_k)\eta_k \Theta_k  = \tilde{y}^k - \eta_{k - 1} \Theta_k$.
By induction, $\tilde{y}^0 = y^0$, and $\Theta_0 = 0$, for all $k\geq 0$, we obtain
\myeq{eq:thm_O1_non_ergo_induction}{
	\tilde{y}^{k + 1} - \eta_k \Theta_{k + 1} = \tilde{y}^1 - \eta_0 \Theta_1 = \tilde{y}^0 - (1 - \tau_0)\eta_0 \Theta_0 = y^0.
}
By the update of $s^{k + 1}$ in \eqref{eq:scheme_s}, the $B_g$-boundedness of $g$, and $0 \in \Kc$, we have
\myeqn{
	\norms{s^{k + 1}} =   \Vert \proj_{\Kc} \big( \tilde{y}^k / \rho_k + g(\hat{x}^k) \big)\Vert 
\leq  \norms{ \tilde{y}^k / \rho_k + g(\hat{x}^k)} \leq  \norms{\tilde{y}^k} / \rho_k + B_g.
}
Furthermore, by the $\Theta_{k + 1}$-update in \eqref{eq:APD_scheme1} and the connection between $y^{k + 1}$ and $s^{k + 1}$ described by \eqref{eq:keylm_yk_in_domH}, we have
\myeq{eq:thm_O1_non_ergo_s}{
	\norms{\Theta_{k + 1}} = \norms{g(x^{k + 1}) - s^{k + 1}} \leq B_g + \norms{s^{k + 1}} \leq 2B_g +  \norms{\tilde{y}^k} / \rho_k.
}
Let us prove the following estimate \eqref{eq:Lk_bdd_d_z} by induction. 
\myeq{eq:Lk_bdd_d_z}{
	\norms{\tilde{y}^k}  \leq \frac{\rho_k}{\gamma} \Big[\frac{\norms{y^0}}{\rho_0} + 2(1 - \gamma)B_g \Big].
}
For $k = 0$, \eqref{eq:Lk_bdd_d_z} holds since $\gamma \in (0, 1)$ and $\tilde{y}^0 = y^0$. 
Suppose that \eqref{eq:Lk_bdd_d_z} holds for all $k \geq 0$, we prove \eqref{eq:Lk_bdd_d_z} for $k+1$.
Using \eqref{eq:thm_O1_non_ergo_induction}, \eqref{eq:thm_O1_non_ergo_s}, $\rho_{k+1} \geq \rho_0$, and the induction hypothesis, we have
\myeqn{
\arraycolsep=0.1em
\begin{array}{lcl}
	\frac{\norms{\tilde{y}^{k + 1}}}{\rho_{k + 1}} & \stackrel{\eqref{eq:thm_O1_non_ergo_induction}} = & 
	\frac{1}{\rho_{k + 1}} \norms{y^0 + \eta_k \Theta_{k + 1}} \stackrel{\eqref{eq:update_rule}}\leq \frac{\norms{y^0}}{\rho_0} + (1 - \gamma)\norms{\Theta_{k + 1}} 	 \stackrel{\eqref{eq:thm_O1_non_ergo_s}}\leq  \frac{\norms{y^0}}{\rho_0} + (1 - \gamma)\left(2B_g + \frac{\norms{\tilde{y}^k}}{\rho_k}\right)\vspace{1ex}\\
	& \leq & \frac{\norms{y^0}}{\rho_0} + (1 - \gamma)\left(2B_g + \frac1\gamma \left[\frac{\norms{y^0}}{\rho_0} + 2(1 - \gamma)B_g\right]\right) =  \frac1\gamma \left[\frac{\norms{y^0}}{\rho_0} + 2(1 - \gamma)B_g\right],
\end{array}
}
which proves \eqref{eq:Lk_bdd_d_z}.

Now, since $\Kc$ is a cone, we have $\prox_{\rho_kH^{*}}(0) = 0$.
Hence, we can show that
\myeq{eq:bound_Lk_ratio1}{
\hspace{-2ex}
\arraycolsep=0.2em
\begin{array}{lcl}
	\Lb_g^k & =  & \Lb_g (y^{k + 1}) \stackrel{\eqref{eq:in_domH}}\leq L_g \norms{\prox_{\rho_k H^{*}} \left(\tilde{y}^k + \rho_k g(\hat{x}^k)\right)} \vspace{1ex}\\
	& =  & L_g \norms{\prox_{\rho_k H^{*}} \left(\tilde{y}^k + \rho_k g(\hat{x}^k)\right) - \prox_{\rho_k H^{*}} (0)}  \leq L_g (2\norms{ \dot{y}} + \norms{\tilde{y}^k + \rho_k g(\hat{x}^k)}) \vspace{1ex}\\
	& \stackrel{\eqref{eq:Lk_bdd_d_z}}\leq & L_g \Big(\rho_k B_g + \frac{\rho_k}\gamma \left[ \norms{y^0}/\rho_0 + 2(1 - \gamma)B_g \right]\Big).
\end{array}
\hspace{-1ex}
}
Therefore, by the new update rule of $L_k$, similar to \eqref{eq:bound_Lk_ratio}, we can easily show that
\myeqn{
\hspace{-2ex}
\arraycolsep=0.1em
\begin{array}{lcl}
	L_k  - {~} \Lb_g^k - L_f - \frac{\rho_k^2 M_g^2}{\rho_k - \eta_k} 
	& \geq &	\frac{\rho_k}\gamma \left[L_g \left(\frac{\norms{y^0}}{\rho_0} + (2 - \gamma)B_g\right) + M_g^2 - L_g \left(\frac{\norms{y^0}}{\rho_0} + (2 - \gamma)B_g\right) - M_g^2\right] \vspace{0.5ex}\\
	& = & 0.
\end{array}
\hspace{-2ex}
}
Using similar proofs as in Theorem \ref{thm:BigO_ergodic_rate} in Appendix \ref{subsec:alg1_thm_ergodic}, we can see that \eqref{eq:ergodic_thm_L} still holds with $\Lc(x, s, y) := F(x) + \iprods{y,\ g(x) - s}$.
Substituting $(x, s) := (x^{\star}, s^{\star})$ into \eqref{eq:ergodic_thm_L}, we get
\myeqn{
	F(\bar{x}^{k}) + \iprods{y,\ g(\bar{x}^{k}) - \bar{s}^{k}} - F^{\star} \leq \frac{\Rc_0^2 (x^{\star}, y)}{2k}, \ \text{where} \ \ \Rc_0^2(x,y) := L_0\norms{x^0 - x}^2 + \frac{1}{\eta_0}\norms{y^0 - y}^2.
}
Let $\Rc^2_0(y) := \Rc_0^2 (x^{\star}, y)$. 
Then, for any fixed $r > 0$, we have
\myeq{eq:thm1_bdd1}{
	F(\bar{x}^{k}) - F^{\star} \leq F(\bar{x}^{k}) - F^{\star} + r\norms{g(\bar{x}^{k}) - \bar{s}^k} \leq \frac{1}{2k} \sup\sets{\Rc^2_0(y):\ \norms{y} \leq r}.
}
On the other hand, by the saddle-point relation \eqref{eq:saddle_pt_s}, we have
\myeqn{
	F(\bar{x}^k) + \iprods{y^{\star},\ g(\bar{x}^k) - \bar{s}^k} = \Lc(\bar{x}^k, \bar{s}^k, y^{\star}) \geq \Lc(x^{\star}, s^{\star}, y^{\star}) = F^{\star}.
}
By the Cauchy-Schwarz inequality, the last estimate leads to
\myeq{eq:thm1_bdd2}{
	F(\bar{x}^{k}) - F^{\star} \geq -\iprods{y^{\star},\ g(\bar{x}^k) - \bar{s}^k} \geq -\norms{y^{\star}} \norms{g(\bar{x}^{k}) - \bar{s}^{k}}.
}
Substituting \eqref{eq:thm1_bdd2} into \eqref{eq:thm1_bdd1}, we get
\myeqn{
	(r - \norms{y^{\star}})\norms{g(\bar{x}^{k}) - \bar{s}^{k}} \leq \dfrac{1}{2k} \sup \sets{\Rc_0^2 (y):\ \norms{y} \leq r}.
}
Let us choose $r := \norms{y^{\star}} + 1$. 
Since $\bar{s}^k \in -\Kc$ due to \eqref{eq:scheme_s}, $\dom{H} = -\Kc$, and $\Kc$ is convex, the last inequality implies that
\myeq{eq:thm1_bdd3}{
\hspace{-2ex}
\arraycolsep=0.1em
\begin{array}{lcl}
	\kdist{-\Kc}{ g(\bar{x}^k)}  & = & \displaystyle\inf_{s \in -\Kc} \norms{g(\bar{x}^k) - s} \leq \norms{g(\bar{x}^{k}) - \bar{s}^{k}}  \leq  \frac{1}{2k} \sup \sets{\Rc_0^2 (y) : \norms{y} \leq \norms{y^{\star}} + 1} \vspace{1ex} \\
	& \leq & \frac{1}{2k} \left[ L_0 \norms{x^0 - x^{\star}}^2 + \frac{1}{\eta_0} {(\norms{y^0} + \norms{y^{\star}} + 1)}^2\right],
\end{array}
\hspace{-4ex}
}
Combining \eqref{eq:thm1_bdd1}, \eqref{eq:thm1_bdd2}, and \eqref{eq:thm1_bdd3}, we arrive at the conclusion in Statement (a).
Statements (b), (c), and (d) can be proven similarly but using the results of Theorems~\ref{thm:alg2_ergodic}, \ref{thm:BigO_nonegrodic_rate}, and  \ref{thm:BigO_nonegrodic_rate_str}, respectively.
We therefore omit the details here without repeating the process.
\Eproof

\bibliographystyle{plain}

\end{document}